\documentclass[10pt]{amsart}
\usepackage{amsmath,amscd,amssymb,amsfonts}
\usepackage{mathrsfs}
\usepackage{hyperref}

\numberwithin{equation}{section}       % Number formulas within sections
\numberwithin{figure}{section}       % Number figures within sections
\theoremstyle{plain}

\newtheorem{prop}{Proposition}[section]

\newtheorem{lemm}[prop]{Lemma}
\newtheorem{fact}[prop]{Fact}
\newtheorem{theoalph}{Theorem}

\newtheorem{coroalph}[theoalph]{Corollary}

\theoremstyle{definition}
\newtheorem{defi}[prop]{Definition}

\theoremstyle{remark}
\newtheorem{rema}[prop]{Remark}

\newtheorem{ques}[prop]{Question}

\newtheoremstyle{citing}% name
  {3pt}%      Space above, empty = `usual value'
  {3pt}%      Space below
  {\itshape}% Body font
  {}%         Indent amount (empty = no indent, \parindent = para indent)
  {\bfseries}% Thm head font
  {.}%        Punctuation after thm head
  {.5em}%     Space after thm head: " " = normal interword space;
  {\thmnote{#3}}% Thm head spec

\theoremstyle{citing}
\newtheorem*{generic}{}% all text supplied in the note

\DeclareMathAlphabet{\mathpzc}{OT1}{pzc}{m}{it} % Zapf Chancery math alphabet

\newcommand{\R}{\mathbb{R}}

\newcommand{\cO}{\mathcal{O}}

\newcommand{\sA}{\mathscr{A}}

\newcommand{\sM}{\mathscr{M}}

\newcommand{\hI}{\widehat{I}}
\newcommand{\hJ}{\widehat{J}}

\newcommand{\hW}{\widehat{W}}

\newcommand{\hlambda}{\widehat{\lambda}}

\newcommand{\tI}{\widetilde{I}}
\newcommand{\tJ}{\widetilde{J}}

\newcommand{\teta}{\widetilde{\teta}}

\newcommand{\partn}[1]{{\smallskip \noindent \textbf{#1.}}}% for parts of a proof
\renewcommand{\=}{: =}
\newcommand{\dd}{\hspace{1pt}\operatorname{d}\hspace{-1pt}}

\DeclareMathOperator{\Lip}{Lip} % best Lipschitz constant norm notation
\DeclareMathOperator{\dist}{dist}

\DeclareMathOperator{\supp}{supp} % Topological support of a measure
\DeclareMathOperator{\Crit}{Crit} % Set of critical points
\DeclareMathOperator{\Leb}{Leb}
\DeclareMathOperator{\Exp}{Exp}
\DeclareMathOperator{\per}{per}

\newcommand{\chiinf}{\chi_{\inf}}
\newcommand{\chiinff}{\chi_{\inf}(f)}
\newcommand{\chiper}{\chi_{\per}}
\newcommand{\chiperf}{\chiper(f)}
\newcommand{\ellmax}{\ell_{\max}}
\newcommand{\hellmax}{\widehat{\ellmax}}

\newcommand{\whf}{\widehat{f}}
\newcommand{\wtf}{\widetilde{f}}

\begin{document}

\title{Asymptotic expansion of smooth interval maps}
\author{Juan Rivera-Letelier}
\address{Department of Mathematics, University of Rochester. Hylan Building, Rochester, NY~14627, U.S.A.}
\email{riveraletelier@gmail.com}
\urladdr{\url{http://rivera-letelier.org/}}

\begin{abstract}
We associate to each non-degenerate smooth interval map a number measuring its global asymptotic expansion.
We show that this number can be calculated in various different ways.
A consequence is that several natural notions of nonuniform hyperbolicity coincide.
In this way we obtain an extension to interval maps with an arbitrary
number of critical points of the remarkable result of Nowicki and Sands
characterizing the Collet-Eckmann condition for unicritical maps.
This also solves a conjecture of Luzzatto in dimension~$1$.

Combined with a result of Nowicki and Przytycki, these considerations imply that several natural nonuniform hyperbolicity conditions are invariant under topological conjugacy.
Another consequence is for the thermodynamic formalism of
non-degenerate smooth interval maps: A non-degenerate smooth map has
a high-temperature phase transition if and only if it is not Lyapunov hyperbolic.
\end{abstract}
% \subjclass[2000]{S99}
% \keywords{Key words}

\maketitle

%
%%% Main
%

\section{Introduction}
In the last few decades, the statistical and stochastic properties of nonuniformly hyperbolic maps have been extensively studied in the one-dimensional setting, see for example~\cite{BruLuzvSt03,GraSmi09,KelNow92,RivShe14,She13,You92} and references therein.
These maps are known to be abundant, see for
example~\cite{AviMor05c,BenCar85,Jak81,GaoShe14,Lyu02,Tsu01b,WanYou06} for interval maps and~\cite{Asp13,Ree86,Smi00,GraSwi00} for complex rational maps.

In this paper we associate to each non-degenerate smooth interval
map a number measuring its global asymptotic expansion.
Our main result is that this number can be calculated in various
different ways.
For example, it can be calculated using the Lyapunov exponents of
periodic points or the Lyapunov exponents of invariant measures, and it can also be calculated using the exponential contraction rate of preimages of a small ball.
This implies that several natural notions of nonuniform hyperbolicity
coincide, including the existence of an absolutely continuous
invariant probability (acip) that is exponentially mixing.
In this way we obtain an extension to interval maps with an arbitrary
number of critical points of the remarkable result of Nowicki and Sands
characterizing the Collet-Eckmann condition for unicritical maps,
see~\cite{NowSan98}.
Moreover, this solves in the affirmative a conjecture of Luzzatto in dimension~$1$, see~\cite[Conjecture~$1$]{Luz06}.

Combined with a result of Nowicki and Przytycki, we obtain that several natural notions of nonuniform hyperbolicity are invariant under topological conjugacy, see~\cite{NowPrz98}.
In particular, for non-degenerate smooth interval maps the existence of an exponentially mixing acip is invariant under topological conjugacy.

Combined with~\cite{Gou05, MelNic05, MelNic09, Tyr05, You99}, these
considerations imply that an arbitrary exponentially mixing
acip satisfies strong statistical properties, such as the local central limit theorem and the vector-valued almost sure invariant principle.
On the other hand, by~\cite{RivShe14} it follows that for some~$p > 1$ the density of such a measure is in the space~$L^p(\Leb)$.

Our main result provides an important step in the study of the
thermodynamic formalism of non-degenerate smooth interval maps in~\cite{PrzRiv1405}.\footnote{The proof of our Main Theorem applies without change to the more general class of maps
  considered in~\cite{PrzRiv1405}, see Theorem~C of that paper.
Note however that, although the proof in~\cite{PrzRiv1405} follows the
proof of our Main Theorem, it has a part that is different.
This modified proof only gives a qualitative version of our Main
Theorem, similar to Corollary~\ref{c:equivalences}.}
Combining our main result with~\cite[Theorem~A]{PrzRiv1405}, we obtain
a characterization of those maps having a high-temperature phase transition.

We proceed to describe our results more precisely.
To simplify the exposition, below we state our results in a more restricted setting than what we are able to handle.
For general versions, see~\S\ref{s:quantified expansion} and the remarks in~\S\ref{ss:proof of corollaries}.

\subsection{Quantifying asymptotic expansion}
\label{ss:quantified expansion}
Let~$I$ be a compact interval and~${f \colon I \to I}$ a smooth map.
A \emph{critical point of~$f$} is a point of~$I$ at which the derivative of~$f$ vanishes.
The map~$f$ is \emph{non-degenerate} if it is non-injective, if the number of its critical points is finite, and if at each critical point of~$f$ some higher order derivative of~$f$ is nonzero.
A non-degenerate smooth interval map is \emph{unicritical} if it has a
unique critical point.\footnote{Note that every unicritical map is
  unimodal, but not conversely.}

Let~$f \colon I \to I$ be a non-degenerate smooth map.
For an integer~$n \ge 1$, a periodic point~$p$ of~$f$ of period~$n$ is \emph{hyperbolic repelling} if~$|Df^n(p)| > 1$.
In this case, denote by
$$ \chi_{p}(f) \=  \frac{1}{n} \ln |Df^n(p)| $$
the Lyapunov exponent of~$p$.
Similarly, for a Borel probability measure~$\nu$ on~$I$ that is invariant by~$f$ denote by
$$ \chi_{\nu}(f) \= \int \ln |Df| \dd \nu $$
its Lyapunov exponent.

The following is our main result.
A non-degenerate smooth map $f \colon I \to I$ is \emph{topologically exact}, if for every open subset~$U$ of~$I$ there is an integer~$n \ge 1$ such that~$f^n(U) = I$.
\begin{generic}[Main Theorem]
For a non-degenerate smooth map~$f \colon I \to I$, the number
$$ \chiperf
\=
\inf \left\{ \chi_p(f) : p \text{ hyperbolic repelling periodic point of } f \right\} $$
is equal to
$$ \chiinff
\=
\inf \left\{ \chi_\nu(f) : \nu \text{ invariant probability measure of } f \right\}. $$
If in addition~$f$ is topologically exact, then there is~$\delta > 0$ such that for every interval~$J$ contained in~$I$ that satisfies~$|J| \le \delta$, we have
$$ \lim_{n \to + \infty} \frac{1}{n} \ln \max \left\{ |W| : W \text{ connected component of } f^{-n}(J) \right\}
=
- \chiinff. $$
Moreover, for each point~$x_0$ in~$I$ we have
\begin{equation}
\label{e:CE2}
\limsup_{n \to + \infty} \frac{1}{n} \ln \min \left\{ |Df^n(x)| : x \in f^{-n}(x_0) \right\}
\le
\chiinff,
\end{equation}
and there is a subset~$E$ of~$I$ of zero Hausdorff dimension such that for each point~$x_0$ in~$I \setminus E$ the $\limsup$ above is a limit and the inequality an equality.
\end{generic}

Except for the equality~$\chiinff = \chiperf$, the hypothesis that~$f$ is topologically exact is necessary, see~\S\ref{ss:notes and references}.

The result above suggests that for a non-degenerate smooth map~$f$ the
number~$\chiperf$ (equal to~$\chiinff$) is a natural measure of the asymptotic expansion of~$f$.
In fact, $\chiinf(f)$ gives a lower bound for the (lower) Lyapunov exponent
of every point in a set of total probability.
This motivates the following definition.
\begin{defi}
\label{d:Lyapunov hyperbolicity}
A non-degenerate smooth map~$f$ is \emph{Lyapunov hyperbolic} if~$\chiinff > 0$.
In this case, we call~$\chiinff$ the \emph{total Lyapunov exponent of~$f$}.
\end{defi}
Lyapunov hyperbolicity can be regarded as a strong form of nonuniform hyperbolicity in the sense of Pesin.
A consequence of the Main Theorem is that Lyapunov hyperbolicity coincides with several natural nonuniform hyperbolicity conditions, see~\S\ref{ss:NUH}.

When restricted to the case where~$f$ is unicritical, the Main Theorem gives a quantified version of the fundamental part of~\cite[Theorem~A]{NowSan98}.
In~\cite[Theorem~A]{NowSan98}, property~\eqref{e:CE2} was only considered in the case where~$x_0$ is the critical point of~$f$; so the assertions concerning~\eqref{e:CE2} in the Main Theorem are new, even when restricted to the case where~$f$ is unicritical.
The proof of~\cite[Theorem~A]{NowSan98} relies heavily on delicate combinatorial arguments that are specific to unicritical maps.
As is, it does not extend to interval maps with several critical points.
When restricted to unicritical maps, our argument is substantially simpler than that of~\cite{NowSan98}.

When~$f$ is a complex rational map, the Main Theorem is the essence of~\cite[Main Theorem]{PrzRivSmi03}.
The proof in~\cite[Main Theorem]{PrzRivSmi03} does not extend to interval maps, because at a key point it relies on the fact that a complex rational map is open as a map of the Riemann sphere to itself.
Our argument allows us to deal with the fact that a non-degenerate smooth interval map is not an open map in general, see~\S\ref{ss:organization} for further details.

\subsection{Nonuniformly hyperbolic interval maps}
\label{ss:NUH}
We introduce some terminology to state a consequence of the Main
Theorem about the equivalence of various nonuniform hyperbolicity conditions.

Let~$(X, \dist)$ be a compact metric space, $T \colon X \to X$ a continuous map and~$\nu$ a Borel probability measure that is invariant by~$T$.
Then~$\nu$ is \emph{exponentially mixing} or \emph{has exponential decay of correlations}, if there are constants~$C > 0$ and~$\rho$ in~$(0, 1)$ such that for every continuous function~$\varphi \colon X \to \R$ and every Lipschitz continuous function~$\psi\colon X \to \R$ we have for every integer~$n \ge 1$
$$ \left| \int_X \varphi \circ f^n \cdot \psi \dd \nu - \int_X \varphi \dd \nu \int_X \psi \dd \nu \right|
\le
C \left( \sup_X |\varphi| \right) \| \psi \|_{\Lip} \rho^n, $$
where~$\| \psi \|_{\Lip} \= \sup_{x, x' \in X, x \neq x'} \frac{|\psi(x) - \psi(x')|}{\dist(x, x')}$.

We denote by~$\Leb$ the Lebesgue measure on~$\R$.
For a non-degenerate smooth map~$f \colon I \to I$, we use \emph{acip} to refer to a Borel probability measure on~$I$ that is absolutely continuous with respect~$\Leb$ and that is invariant by~$f$.

A non-degenerate smooth map~$f \colon I \to I$ has \emph{Uniform
  Hyperbolicity on Periodic Orbits}, if~$\chiperf > 0$.
Moreover, $f$ satisfies the:
\begin{itemize}
\item
\emph{Collet-Eckmann condition}, if all the periodic points of~$f$ are hyperbolic repelling and if for every critical value~$v$ of~$f$ we have
$$ \liminf_{n \to + \infty} \frac{1}{n} \ln |Df^n(v)| > 0. $$
\item
\emph{Backward} or \emph{Second Collet-Eckmann condition at a point~$x$ of~$I$}, if there are constants~$C > 0$ and~$\lambda > 1$, such that for every integer~$n \ge 1$ and every point~$y$ of~$f^{-n}(x)$ we have~$|Df^n(y)| \ge C \lambda^n$.
\item
\emph{Backward} or \emph{Second Collet-Eckmann condition}, if~$f$ satisfies the Backward Collet-Eckmann condition at each of its critical points.
\item
\emph{Exponential Shrinking of Components condition}, if there are
constants~$\delta > 0$ and~$\lambda > 1$ such that for every
interval~$J$ contained in~$I$ that satisfies~$|J| \le \delta$, the
following holds: For every integer~$n \ge 1$ and every connected
component~$W$ of~$f^{-n}(J)$ we have $|W| \le \lambda^{-n}$.
\end{itemize}

In the statement of the following corollary we use the following fact: Every non-degenerate smooth map that is topologically exact has strictly positive topological entropy and a unique measure of maximal entropy, see for example~\cite[\S$3$]{Bal00b}.
Finally, a measure~$\rho$ on~$I$ has a \emph{power-law lower bound}, if there are constants~$C > 0$ and~$\alpha > 0$ such that for every interval~$J$ contained in~$I$ we have~$\rho(J) \ge C |J|^\alpha$.

\begin{coroalph}
\label{c:equivalences}
For a non-degenerate smooth map~$f \colon I \to I$ that is topologically exact, the following properties are equivalent:
\begin{enumerate}
\item[1.]
Lyapunov hyperbolicity ($\chiinff > 0$).
\item[2.]
Uniform Hyperbolicity on Periodic Orbits ($\chiperf > 0$).
\item[3.]
Existence of an exponentially mixing acip for~$f$.
\item[4.]
The map~$f$ is conjugated to a piecewise affine and expanding multimodal map by a bi-H{\"o}lder continuous function.
\item[5.]
The map~$f$ satisfies the Exponential Shrinking of Components condition.
\item[6.]
The map~$f$ satisfies the Backward Collet-Eckmann condition at some
point of~$I$.
\item[7.]
The maximal entropy measure of~$f$ has a power-law lower bound.
\end{enumerate}
Furthermore, these equivalent conditions are satisfied when~$f$ satisfies the Collet-Eckmann or the Backward Collet-Eckmann condition.
\end{coroalph}

The equivalence $1 \Leftrightarrow 3$ solves \cite[Conjecture~$1$]{Luz06} in dimension~$1$.

When~$f$ is unicritical, the equivalence of conditions~$1$--$5$ was proved by~Nowicki and Sands in~\cite[Theorem~A]{NowSan98}.
They also showed, still in the case where~$f$ is unicritical, that the Collet-Eckmann and the Backward Collet-Eckmann conditions are equivalent and that each of these conditions is equivalent to conditions~$1$--$5$.
In contrast, for maps with several critical points the Collet-Eckmann and the Backward Collet-Eckmann conditions are not equivalent and neither of these conditions is equivalent to conditions~$1$--$7$, see~\cite[\S$6$]{PrzRivSmi03}.
When~$f$ is a complex rational map, a statement analog to
Corollary~\ref{c:equivalences} was shown by Przytycki, Smirnov, and
the author in~\cite[Main Theorem]{PrzRivSmi03},\footnote{In~\cite{PrzRivSmi03} condition~$4$ was
  interpreted as the existence of a ``H{\"o}lder coding tree.''} \cite[Corollary~$1.1$]{PrzRiv07} and~\cite[Theorem~B]{Riv10}.

Even when restricted to the case where~$f$ is unicritical, the
implication~$6 \Rightarrow 5$ of Corollary~\ref{c:equivalences} is new.
It is the main new ingredient of the proof, which is provided by Main Theorem.
The implication~$5 \Rightarrow 4$ is also new.
The rest of the implications are known, or can be easily adapted from known properties of unicritical interval maps or complex rational maps, see~\S\ref{ss:proof of corollaries} for references.

\subsection{Exponentially mixing acip's}
\label{ss:exponentially mixing acips}
Let~$f \colon I \to I$ be a non-degenerate smooth map that is topologically
exact and that is Lyapunov hyperbolic.
Such a map has a unique exponentially mixing acip.
In~\cite[Theorem~C]{PrzRiv07}, this measure is constructed using the general method of Young in~\cite{You99}.\footnote{The proof of~\cite[Theorem~C]{PrzRiv07} is written for complex rational maps and applies without change to topologically exact non-degenerate smooth interval maps. See~\cite[Corollary~$2.19$]{RivShe14} for a proof written for interval maps.}
When a measure~$\nu$ on~$I$ can be obtained in this way, we say \emph{$\nu$ can be obtained through a Young tower with an exponential tail estimate}.
Such a measure has several statistical properties, including the ``local central limit theorem'' and the ``vector-valued almost sure invariant principle,'' see~\cite{MelNic09, You99} for these results and for precisions, and~\cite{Gou05, MelNic05, Tyr05} for other statistical properties satisfied by such a measure.

On the other hand, for~$f$ as above there is~$p(f) > 1$ with the
following property: For~$p \ge 1$ the density of the unique exponentially mixing acip of~$f$ is in the space~$L^p(\Leb)$ if~$1 \le p < p(f)$, and it is not in~$L^p(\Leb)$ if~$p > p(f)$.
See~\cite[Corollary~$2.19$]{RivShe14}, where a geometric characterization of~$p(f)$ is also given.\footnote{If~$f$ is unicritical and we denote its critical point by~$c$, then~$p(f) = \ell_c / (\ell_c - 1)$.}

In view of the results above, the following corollary is a direct consequence of Corollary~\ref{c:equivalences} and of general properties of non-degenerate smooth interval maps.
\begin{coroalph}
  \label{c:exponentially mixing acips}
Let~$f$ be a non-degenerate smooth interval map having an exponentially mixing acip~$\nu$.
Then there is~$p > 1$ such that the density of~$\nu$ with respect to~$\Leb$ is in the space~$L^p(\Leb)$.
Moreover, $\nu$ can be obtained through a Young tower with an exponential tail estimate.
In particular, $\nu$ satisfies the local central limit theorem and the vector-valued almost sure invariant principle.
\end{coroalph}

Alves, Freitas, Luzzatto, and Vaienti showed under mild assumptions
that in any dimension each polynomially mixing or stretch
exponentially mixing acip can be obtained through a Young tower with the corresponding tail estimates, see~\cite[Theorem~C]{AlvFreLuzVai11}.
In contrast with this last result, in Corollary~\ref{c:exponentially mixing acips} the existence of~$p > 1$ for which the density of~$\nu$ is in~$L^p(\Leb)$ is obtained as a consequence, and not as a hypothesis.
So the following question arises naturally.
\begin{ques}
Let~$f$ be a non-degenerate smooth interval map having an acip~$\nu$.
Does there exist~$p > 1$ such that the density of~$\nu$ with respect to~$\Leb$ is in the space~$L^p(\Leb)$?
\end{ques}

\subsection{Topological invariance}
\label{ss:topological invariance}
A direct consequence of Corollary~\ref{c:equivalences} and a result of
Nowicki and Przytycki in~\cite{NowPrz98}, is that each of the
conditions~$1$--$7$ of Corollary~\ref{c:equivalences} is invariant
under topological conjugacy for maps having all of its periodic points hyperbolic repelling.
To state this result more precisely, we recall the definition of the ``Topological Collet-Eckmann condition'' introduced in~\cite{NowPrz98}.
Let~$f \colon I \to I$ be a non-degenerate smooth map that is topologically exact and fix~$r > 0$.
Given an integer~$n \ge 1$, the \emph{criticality of~$f^n$ at a point~$x$ of~$I$} is the number of those~$j$ in~$\{0, \ldots, n - 1 \}$ such that the connected component of~$f^{-(n - j)}(B(f^n(x), r))$ containing~$f^j(x)$ contains a critical point of~$f$.
Then~$f$ satisfies the \emph{Topological Collet-Eckmann (TCE) condition}, if for some choice of~$r > 0$ there are constants~$D \ge 1$ and~$\theta$ in~$(0, 1)$, such that the following property holds: For each point~$x$ in~$I$ the set~$G_x$ of all those integers~$m \ge 1$ for which the criticality of~$f^m$ at~$x$ is less than or equal to~$D$, satisfies
$$ \liminf_{n \to + \infty} \frac{1}{n} \# \left( G_x \cap \{1, \ldots, n \} \right)
\ge
\theta. $$

One of the main features of the TCE condition, which is readily seen from its definition, is that it is invariant under topological conjugacy preserving critical points: If~$f \colon I \to I$ is a non-degenerate smooth map satisfying the TCE condition and~$\wtf \colon \tI \to \tI$ is a non-degenerate smooth map that is topologically conjugated to~$f$ by a map preserving critical points, then~$\wtf$ also satisfies the TCE condition.
Nowicki and Przytycki showed in~\cite{NowPrz98} that for a
non-degenerate smooth interval map~$f$, condition~$5$ of
Corollary~\ref{c:equivalences} implies the TCE condition.
They also proved that if in addition all the periodic points of~$f$
are hyperbolic repelling, then the TCE condition implies condition~$2$ of Corollary~\ref{c:equivalences}.
Thus, the following is a direct consequence of Corollary~\ref{c:equivalences} and~\cite{NowPrz98}.
\begin{coroalph}
\label{c:topological invariance}
For a non-degenerate smooth interval map that is topologically exact
and that only has hyperbolic repelling periodic points, the Topological Collet-Eckmann condition is equivalent to each of the conditions~$1$--$7$ of Corollary~\ref{c:equivalences}.
In particular, each of the conditions~$1$--$7$ of Corollary~\ref{c:equivalences} is invariant under topological conjugacy preserving critical points, for maps having only hyperbolic repelling periodic points.
\end{coroalph}
Combining~\cite{NowPrz98} and~\cite[Theorem~A]{NowSan98}, it follows that for unicritical maps having only hyperbolic repelling periodic points the Collet-Eckmann and the Backward Collet-Eckmann conditions are both invariant under topological conjugacy preserving critical points.
This is not the case for maps with several critical points, see~\cite[Appendix~C]{PrzRivSmi03}.

The following is for maps that are not necessarily topologically
exact.
It is obtained by combining Corollary~\ref{c:topological invariance}
with general properties of non-degenerate smooth interval maps, see~\S\ref{ss:proof of corollaries} for the proof.
\begin{coroalph}
\label{c:topological invariance of exponential mixing}
For non-degenerate smooth interval maps having only hyperbolic
repelling periodic points, the property that an iterate has an exponentially mixing acip is invariant under topological conjugacy preserving critical points.
\end{coroalph}
\subsection{High-temperature phase transitions}
\label{ss:phase transitions}
Corollary~\ref{c:equivalences} has a very useful application to the thermodynamic formalism of interval maps, that we proceed to describe.
Let~$f \colon I \to I$ be a non-degenerate smooth interval map that is topologically exact.
Denote by~$\sM(I, f)$ the space of Borel probability measures on~$I$ that are invariant by~$f$.
For a measure~$\nu$ in~$\sM(I, f)$, denote by~$h_{\nu}(f)$ the measure-theoretic entropy of~$f$ with respect to~$\nu$ and for each real number~$t$ put
$$ P(t)
\=
\sup \left\{ h_{\nu}(f) - t \chi_\nu(f) : \nu \in \sM(I, f) \right\}. $$
Combining Ruelle's inequality in~\cite{Rue78} with the fact that the Lyapunov
exponent of every measure in~$\sM(I, f)$ is nonnegative,
see~\cite[Theorem~B]{Prz93} or~Proposition~\ref{p:Lyapunov are
  nonnegative}, it follows that the number above is finite and that
the function~$P \colon \R \to \R$ so defined is convex and nonincreasing.
Moreover, $P$ has at least one zero and that its first zero is
in~$(0, 1]$.
The function~$P$ is called the \emph{geometric pressure function
  of~$f$}, and it is related to various multifractal spectra and large
deviation rate functions associated to~$f$.

Following the usual terminology in statistical mechanics, for a real number~$t_*$ we say~$f$ has a \emph{phase transition at~$t_*$}, if~$P$ is not real analytic at~$t = t_*$.
In accordance with the usual interpretation of~$t > 0$ as the inverse of the temperature in statistical mechanics, if in addition~$t_* > 0$ and~$t_*$ is less than or equal to the first zero of~$P$, then we say that~$f$ has a \emph{high-temperature phase transition}.

The following is an easy consequence of Corollary~\ref{c:equivalences}
and~\cite[Theorem~A]{PrzRiv1405}, see~\S\ref{ss:proof of corollaries}
for the proof.
\begin{coroalph}
\label{c:phase transitions}
For a non-degenerate smooth interval map~$f$ that is topologically exact, the following properties are equivalent:
\begin{enumerate}
\item[1.]
The map~$f$ has a high-temperature phase transition.
\item[2.]
If we denote by~$t_0$ the first zero of~$P$, then for every~$t \ge t_0$ we have~$P(t) = 0$.
\item[3.]
The function~$P$ is nonnegative.
\item[4.]
The map~$f$ is not Lyapunov hyperbolic.
\end{enumerate}
\end{coroalph}
When~$f$ is a complex rational map, the equivalence of conditions~$2$--$4$ is part of~\cite[Main Theorem]{PrzRivSmi03}.\footnote{It is unclear to us if condition~$1$ is equivalent to~$2$--$4$ in the complex setting.}

\subsection{Notes and references}
\label{ss:notes and references}
If the map~$f$ is not topologically exact, then by the Main Theorem we have~$\chiinff = \chiperf$, but the remaining assertions of the Main Theorem do not hold in general.
For an example, consider the logistic map with the Feigenbaum combinatorics, $f_0$.
For this map we have~$\chiinf(f_0) = 0$.
However, if~$J$ is a small closed interval that is disjoint from the post-critical set of~$f_0$, then the limit in the Main Theorem is strictly negative.
Similarly, for every point~$x_0$ that is not in the post-critical set of~$f_0$, the~$\limsup$ in the Main Theorem is strictly positive.
This also shows that the implication~$6 \Rightarrow 1$ of Corollary~\ref{c:equivalences} does not hold for~$f_0$.
Note also that an infinitely renormalizable map~$f$ cannot satisfy any of the conditions~$1$--$5$ of Corollary~\ref{c:equivalences}.

See~\cite{Mih0810} for further examples illustrating the difference
between the Collet-Eckmann condition and conditions~$1$--$7$ of Corollary~\ref{c:equivalences} for maps with at least~$2$ critical points.

Li~\cite{Li17} and Luzzatto and Wang~\cite{LuzWan06} showed that the Collet-Eckmann condition together with a slow recurrence condition is invariant under topological conjugacy preserving critical points.
See also~\cite{LiShe13} for a recent related result.

See~\cite{CorRiv13,CorRiv15b} and references therein for results on low-temperature phase transitions; that is, phase transitions that occur after the first zero of the geometric pressure function.
\subsection{Strategy and organization}
\label{ss:organization}
To prove the Main Theorem and Corollary~\ref{c:equivalences} we follow the structure of the proof of the analog result for complex rational maps in~\cite[Main Theorem]{PrzRivSmi03}.
The main difficulty is the proof that~$\chiperf > 0$ implies the last statement of the Main Theorem, which is essentially the implication~$2 \Rightarrow 5$ of Corollary~\ref{c:equivalences}.
The proof of this fact in~\cite{PrzRivSmi03} relies in an essential way on the fact that a nonconstant complex rational maps is open as a map from the Riemann sphere to itself.
The argument provided here allows us to deal with the fact that a multimodal map is not an open map in general.
Ultimately, it relies on the fact that the boundary of a bounded interval in~$\R$ is reduced to~$2$ points.

To prove implication~$2 \Rightarrow 5$ of Corollary~\ref{c:equivalences} we first remark that the proof of the implication~$2 \Rightarrow 6$ for rational maps in~\cite{PrzRivSmi03} applies without change to interval maps.
Our main technical result is a quantified version of the implication~$6 \Rightarrow 5$ for interval maps.
This is stated as Proposition~\ref{p:CE2 implies ESC}, after some preliminary considerations in~\S\ref{s:preliminaries}.
Its proof occupies all of~\S\ref{s:CE2 implies ESC}.
In~\S\ref{s:quantified expansion} we formulate a strengthened version of the Main Theorem, stated as the Main Theorem', and we deduce it from Proposition~\ref{p:CE2 implies ESC} and known results.
In the proof we use that the Lyapunov exponent of every invariant measure supported on the Julia set is nonnegative~\cite[Theorem~B]{Prz93}.
We provide a simple proof of this fact (Proposition~\ref{p:Lyapunov are nonnegative} in Appendix~\ref{s:Lyapunov are nonnegative}), which holds for a general continuously differentiable interval map.
This result is used again in the proof of Corollary~\ref{c:phase transitions}.

The proofs of Corollaries~\ref{c:equivalences}, \ref{c:topological invariance of exponential mixing}, and~\ref{c:phase transitions} are given in~\S\ref{ss:proof of corollaries}, after we prove the implication~$5 \Rightarrow 4$ of Corollary~\ref{c:equivalences} in~\S\ref{s:conjugacy to affine}.

\subsection{Acknowledgments}
\label{ss:acknowledgments}
I would like to thank the referee for several valuable comments.

This article was completed while the author was visiting Brown University and the Institute for Computational and Experimental Research in Mathematics (ICERM).
The author thanks both of these institutions for the optimal working
conditions provided, and acknowledges partial support from FONDECYT grant 1100922, Chile, and NSF grant DMS-1700291, U.S.A.

\section{Preliminaries}
\label{s:preliminaries}
Throughout the rest of this paper~$I$ denotes a compact interval of~$\R$.
We endow~$I$ with the distance~$\dist$ induced by the absolute value~$| \cdot |$ on~$\R$.
For~$x$ in~$I$ and~$r > 0$, we denote by~$B(x, r)$ the open ball of~$I$ centered at~$x$ and of radius~$r$.
For an interval~$J$ contained in~$I$, we denote by~$|J|$ its length and for~$\eta > 0$ we denote by~$\eta J$ the open interval of~$\R$ of length~$\eta |J|$ that has the same middle point as~$J$.

Given a map~$f \colon I \to I$, a subset~$J$ of~$I$ is \emph{forward invariant} if~$f(J) = J$ and it is \emph{completely invariant} if~$f^{-1}(J) = J$.
\subsection{Fatou and Julia sets}
\label{ss:multimodal}
Following~\cite{dMevSt93}, in this section we introduce the Fatou and Julia sets of a multimodal map and gather some of their basic properties.

A non-injective continuous map $f \colon I \to I$ is \emph{multimodal}, if there is a finite partition of~$I$ into intervals on each of which~$f$ is injective.
A \emph{turning point} of a multimodal map~$f \colon I \to I$ is a point in~$I$ at which~$f$ is not locally injective.

Fix a multimodal map~$f \colon I \to I$.
The \emph{Fatou set~$F(f)$ of~$f$} is the largest open subset of~$I$ on which the iterates of~$f$ form a normal family.
A connected component of~$F(f)$ is called \emph{Fatou component of~$f$}.
A Fatou component~$U$ of~$f$ is \emph{periodic} if for some integer~$p \ge 1$ we have~$f^p(U) \subset U$.
In this case the least integer~$p$ with this property is the \emph{period of~$U$}.

The \emph{Julia set~$J(f)$ of~$f$} is the complement of~$F(f)$ in~$I$.
By definition we have~$f^{-1}(F(f)) \subset F(f)$ and
therefore~$f(J(f)) \subset J(f)$.
In contrast with the complex setting, the Julia set of~$f$  might be empty, reduced to a single point, or might not be
completely invariant.
If the Julia set of~$f$ is not completely invariant, then
it is possible to make an arbitrarily small smooth perturbation of~$f$
outside a neighborhood of~$J(f)$, so that the Julia set of the
perturbed map is completely invariant and coincides with~$J(f)$.

\subsection{Topological exactness}
\label{ss:topological exactness}
Fix a multimodal map~$f \colon I \to I$.
We say that~$f$ is \emph{boundary anchored} if~$f(\partial I)
\subset \partial I$ and that~$f$ is \emph{topologically exact
  on~$J(f)$}, if~$J(f)$ is not reduced to a point and if for every open
subset~$U$ of~$I$ intersecting~$J(f)$ an iterate of~$f|_{J(f)}$
maps~$U \cap J(f)$ onto~$J(f)$.

Since it is too restrictive for our applications to assume that a multimodal map is at the same time boundary anchored and topologically exact on its Julia set, we introduce the following terminology.
We say that a multimodal map~$f$ is \emph{essentially topologically exact on~$J(f)$}, if there is a compact interval~$I_0$ contained in~$I$ that contains all the critical points of~$f$ and such that the following properties hold: $f(I_0) \subset I_0$, the multimodal map~$f|_{I_0} \colon I_0 \to I_0$ is topologically exact on~$J(f|_{I_0})$, and~$\bigcup_{n = 0}^{+ \infty} f^{-n}(I_0)$ contains an interval whose closure contains~$J(f)$.

\subsection{Differentiable interval maps}
\label{ss:differentiable}
Fix a differentiable map~$f \colon I \to I$.

A \emph{critical point of~$f$} is a point at which the derivative of~$f$ vanishes.
A \emph{critical value of~$f$} is the image by~$f$ of a critical point. 
We denote by~$\Crit(f)$ the set of critical points of~$f$.
If~$f$ is in addition a multimodal map, then we put
$$ \Crit'(f) \= \Crit(f) \cap J(f). $$

Let~$J$ be an interval contained in~$I$ and let~$n \ge 1$ be an integer.
Then each connected component of~$f^{-n}(J)$ is a \emph{pull-back of~$J$ of order~$n$}, or just a \emph{pull-back of~$J$}.
If in addition~$f^n \colon W \to J$ is a diffeomorphism, then the pull-back~$W$ is \emph{diffeomorphic}.
Note that if~$f$ is boundary anchored and~$W$ is a pull-back of~$J$ of order~$n$, then~$f^n(\partial W) \subset \partial J$.

Let~$J$ be an interval contained in~$I$, let~$n \ge 1$ be an integer, and let~$W$ be a pull-back of~$J$ by~$f^n$.
We say~$W$ is a \emph{child of~$J$},\footnote{This definition is a variant of the usual definition of ``child.'' It is adapted to deal with the case where~$f$ has a critical connection.} if~$W$ contains a unique critical point~$c$ of~$f$ in~$J(f)$ and if there is~$s$ in~$\{0, \ldots, n - 1 \}$ such that~$f^s(c)$ belongs to~$\Crit(f)$ and such that the following properties hold:
\begin{enumerate}
\item[1.]
Either~$s = n - 1$ or the pull-back of~$J$ by~$f^{n - s - 1}$ containing~$f^{s + 1}(c)$ is diffeomorphic.
\item[2.]
For each~$s'$ in~$\{0, \ldots, s \}$ the pull-back of~$J$ by~$f^{n - s'}$ containing~$f^{s'}(c)$ is either disjoint from~$\Crit(f)$ or~$f^{s'}(c)$ belongs to~$\Crit(f)$ and then~$f^{s'}(c)$ is the unique critical point of~$f$ contained in this set.
\end{enumerate}

\subsection{Interval maps of class~$C^3$ with non-flat critical points}
A differentiable interval map~$f \colon I \to I$ is \emph{of class~$C^3$
  with non-flat critical points}, if:
\begin{itemize}
\item
The set~$\Crit(f)$ is finite and~$f$ is of class~$C^3$ outside $\Crit(f)$.
\item
For each critical point~$c$ of~$f$ there exists a number $\ell_c>1$ and diffeomorphisms~$\phi$ and~$\psi$ of~$\R$ of class~$C^3$, such that $\phi(c)=\psi(f(c))=0$ and such that on a neighborhood of~$c$ on~$I$ we have,
$$ |\psi\circ f| = |\phi|^{\ell_c}. $$
The number~$\ell_c$ is the \emph{order of~$f$ at~$c$}.
\end{itemize}

Denote by~$\sA$ the collection of non-injective interval maps of
class~$C^3$ with non-flat critical points, whose Julia set is
completely invariant and contains at least~$2$ points.
Note that every smooth non-degenerate interval map that is
topologically exact is in~$\sA$, and that every interval map in~$\sA$ is a continuously differentiable multimodal map.

We use the following important fact: For each map in~$\sA$ every Fatou
component is mapped to a periodic Fatou component under forward
iteration, and the number of periodic Fatou components is finite,
see~\cite[Chapter~IV, Theorem~AB]{dMevSt93}.

The following version of the Koebe principle follows from~\cite[Theorem~C($2$)(ii)]{vStVar04}.
As for non-degenerate smooth interval maps, a periodic point~$p$ of period~$n$ of a map~$f$ in~$\sA$ is \emph{hyperbolic repelling} if~$|Df^n(p)| > 1$.
\begin{lemm}[Koebe principle]
\label{l:Koebe principle}
Let~$f \colon I \to I$ be an interval map in~$\sA$ all whose periodic points in~$J(f)$ are hyperbolic repelling.
Then there is~$\delta_0 > 0$ such that for every~$K > 1$ there is~$\varepsilon$ in~$(0, 1)$ such that the following property holds.
Let~$J$ be an interval contained in~$I$ that intersects~$J(f)$ and satisfies~$|J| \le \delta_0$.
Moreover, let~$n \ge 1$ be an integer and $W$ a diffeomorphic pull-back of~$J$ by~$f^n$.
Then for every~$x$ and~$x'$ in the unique pull-back of~$\varepsilon J$ by~$f^n$ contained in~$W$ we have
$$ K^{-1} \le |Df^n(x)| / |Df^n(x')| \le K. $$
\end{lemm}

The following general fact is used in the proof of the Main Theorem' in~\S\ref{s:quantified expansion}.
\begin{fact}
\label{f:positive entropy}
If~$f$ is an interval map in~$\sA$ that is topologically exact on~$J(f)$, then~$J(f)$ contains a uniformly expanding set whose topological entropy is strictly positive.
In particular, the Hausdorff dimension of~$J(f)$ is strictly positive.
\end{fact}

The following lemma is standard, see for example~\cite{Riv1206} for part~1.
\begin{lemm}
  \label{l:pull-backs}
Let~$f \colon I \to I$ be a multimodal map in~$\sA$ having all of its periodic points in~$J(f)$ hyperbolic repelling.
Then the following properties hold.
\begin{enumerate}
\item[1.]
For every integer~$n \ge 1$, every pull-back~$W$ of~$B(x, \delta_1)$ by~$f^n$ intersects~$J(f)$, contains at most~$1$ critical point of~$f$, and is disjoint from~$(\Crit(f) \cup \partial I) \setminus J(f)$.
\item[2.]
  For every~$\kappa > 0$ there is~$\delta_2 > 0$ such that for every~$x$ in~$J(f)$, every integer~$n \ge 1$, and every
  pull-back~$W$ of~$B(x, \delta_2)$ by~$f^n$, we have~$|W| \le \kappa$.
\end{enumerate}
\end{lemm}

\section{Exponential shrinking of components}
\label{s:CE2 implies ESC}
The purpose of this section is to prove the following proposition.
It is the key step in the proof of the Main Theorem, which is given in the next section.

\begin{prop}
\label{p:CE2 implies ESC}
Let~$f \colon I \to I$ be a map in~$\sA$ that is topologically exact on~$J(f)$.
Suppose there is a point~$x_0$ of~$J(f)$ and constants~$C > 0$ and~$\lambda > 1$ such that for every integer~$n \ge 1$ and every point~$x$ in~$f^{-n}(x_0)$ we have
$$ |Df^n(x)| \ge C \lambda^n. $$
Then every periodic point of~$f$ in~$J(f)$ is hyperbolic repelling and for every~$\lambda_0$ in~$(1, \lambda)$ there is a constant~$\delta_2 > 0$ such that the following property holds.
Let~$J$ be an interval contained in~$I$ that intersects~$J(f)$ and satisfies~$|J| \le \delta_2$.
If~$J(f)$ is not an interval, then assume that~$J$ is not a neighborhood of a periodic point in the boundary of a Fatou component of~$f$.\footnote{There is an example showing that this hypothesis is necessary, see~\cite[Proposition~A]{Riv1206}. However, a qualitative result holds when this hypothesis is not satisfied, see~\cite[Theorem~B]{Riv1206}.}
Then for every integer~$n \ge 1$ and every pull-back~$W$ of~$J$ by~$f^n$, we have
\begin{equation}
  \label{e:weak ESC}
|W| \le \lambda_0^{- n}.
\end{equation}
\end{prop}

The proof of this proposition is at the end of this section.
It is based on several lemmas.

In this section, a critical point~$c$ of a map~$f$ in~$\sA$ is \emph{exposed}, if for every integer~$j \ge 1$ the point~$f^j(c)$ is not a critical point of~$f$.
Given~$c$ in~$\Crit'(f)$, let~$s \ge 0$ be the largest integer such that~$f^s(c)$ is in~$\Crit(f)$ and put
$$ \widehat{\ell_c} \= \prod_{\substack{j \in \{0, \ldots, s \} \\ f^j(c) \in \Crit(f)}} \ell_{f^j(c)}
\text{ and }
\hellmax \= \max \left\{ \widehat{\ell_c} : c \in \Crit'(f) \right\}. $$

\begin{lemm}
\label{l:child}
Let~$f \colon I \to I$ be an interval map in~$\sA$ such that all of its periodic points in~$J(f)$ are hyperbolic repelling.
Then there are~$\delta_3 > 0$ and $C_1 > 1$ such that for every interval~$J$ that intersects~$J(f)$ and satisfies~$|J| \le \delta_3$ and~$C_1 J \subset I$, the following property holds: For every integer~$n \ge 1$ and every pull-back~$W$ of~$J$ by~$f^n$ such that the pull-back of~$C_1 J$ by~$f^n$ containing~$W$ is a child of~$C_1 J$, we have
$$ |W|
\le
6 \hellmax |J| \max \left\{ |Df^n(a)| : a \in \partial W \right\}^{-1}. $$
\end{lemm}
\begin{proof}
Let~$\delta_0 > 0$ and~$\varepsilon$ in~$(0, 1)$ be given by Lemma~\ref{l:Koebe principle} with~$K = 2$ and let~$\delta_1 > 0$ be given by Lemma~\ref{l:pull-backs}.
Since the critical points of~$f$ are non-flat, there is~$\delta_* > 0$ so that for each~$c$ in~$\Crit'(f)$, each integer~$s \ge 0$ such that~$f^s(c)$ is in~$\Crit'(f)$, and each interval~$W$ contained in~$B(c, \delta_*)$ we have
$$ |W| \max \left\{ |Df^{s + 1}(a)| : a \in \partial W \right\}
\le
3 \widehat{\ell_{c}} |f^{s + 1}(W)|. $$
Let~$\delta_2 > 0$ be given by Lemma~\ref{l:pull-backs}(2) with~$\kappa = \delta_*$.

We prove the lemma with~$\delta_3 = \varepsilon \min \{ \delta_2, \delta_0 \}$ and~$C_1 = \varepsilon^{-1}$.
To do this, let~$J$ be an interval contained in~$I$ that intersects~$J(f)$ and satisfies
\begin{displaymath}
  |J| \le \delta_2
  \text{ and }
  \hJ \= \varepsilon^{-1} J \subset I.
\end{displaymath}
Moreover, let~$n \ge 1$ be an integer and let~$W$ be a pull-back of~$J$ by~$f^n$ such that the pull-back~$\hW$ of~$\hJ$ by~$f^n$ containing~$W$ is a child of~$\hJ$.
Let~$c$ be the unique critical point of~$f$ contained in~$\hW$ and let~$s$ be the largest element of~$\{0, \ldots, n - 1 \}$ such that~$f^s(c)$ is in~$\Crit(f)$.
So either~$s = n - 1$ or the pull-back~$\hW'$ of~$\hJ$ by~$f^{n - s - 1}$ containing~$f^{s + 1}(W)$ is diffeomorphic.
Then the Koebe principle (Lemma~\ref{l:Koebe principle}) implies that, if we denote by~$W'$ the pull-back of~$J$ by~$f^{n - s - 1}$ containing~$f^{s + 1}(W)$, then
$$ |W'|
\le
2 |J| \max \left\{ |Df^{n - s - 1}(a')| : a' \in \partial W' \right\}^{-1}. $$
On the other hand, by our choice of~$\delta_2$ we have~$W \subset \hW \subset B(c, \delta_*)$, so by our choice of~$\delta_*$ we have
\begin{multline*}
|W|
\le
3 \widehat{\ell_{c}} |f^{s + 1}(W)| \max \left\{ |Df^{s + 1}(a)| : a \in \partial W \right\}^{-1}
\\ \le
3 \hellmax |W'| \max \left\{ |Df^{s + 1}(a)| : a \in \partial W \right\}^{-1}.
\end{multline*}
The desired inequality is obtained by combining the last~$2$ displayed inequalities.
\end{proof}

\begin{lemm}
\label{l:maximum principle}
Let~$f \colon I \to I$ be an interval map in~$\sA$ such that all of its periodic points in~$J(f)$ are hyperbolic repelling.
Suppose that none of the boundary points of~$I$ is a critical point of~$f$ and let~$C_1 > 1$ be the constant given by Lemma~\ref{l:child}.
Then, for every~$\eta > 1$ there is a constant~$\delta(\eta) > 0$ such that for every interval~$\hJ$ that intersects~$J(f)$ and satisfies~$|\hJ| \le \delta(\eta)$ and~$C_1 \hJ \subset I$, the following properties hold for every integer~$n \ge 1$ and every pull-back~$\hW$ of~$\hJ$ by~$f^n$:
\begin{enumerate}
\item[1.]
For every interval~$J$ contained in~$\hJ$, the number of pull-backs of~$J$ by~$f^n$ contained in~$\hW$ is bounded from above by~$2 \eta^n$.
\item[2.]
$|\hW|
\le
12 \hellmax \eta^n |\hJ| \max \left\{ |Df^n(a)| : a \in \partial \hW \right\}^{-1}$.
\end{enumerate}
\end{lemm}
\begin{proof}
Let~$\delta_0 > 0$ and~$\varepsilon$ in~$(0, 1)$ be given by Lemma~\ref{l:Koebe principle} with~$K = 2$, let~$\delta_1 > 0$ be given by Lemma~\ref{l:pull-backs}(1), and let~$\delta_3 > 0$ and~$C_1 > 1$ be given by Lemma~\ref{l:child}.
Enlarging~$C_1$ if necessary we assume~$C_1 \ge \varepsilon^{-1}$.
On the other hand, let~$L \ge 1$ be a sufficiently large integer such that $\eta^L > 6\hellmax$ and let~$\delta_* > 0$ be sufficiently small so that for every exposed critical point~$c$ of~$f$ and every~$j$ in~$\{0, \ldots, L \}$, the point~$f^j(c)$ is not in~$B(\Crit(f), \delta_*)$.
Finally, let~$\delta_2$ be given by Lemma~\ref{l:pull-backs}(2) with
\begin{displaymath}
  \kappa
  \=
C_1^{-1} \min \left\{ \delta_0, \delta_1, \delta_3, \delta_*, \dist(\Crit(f), \partial I) \right\}.
\end{displaymath}

We prove the lemma with~$\delta(\eta) = \delta_2$.
To do this, let~$\hJ$ be an interval that intersects~$J(f)$ and satisfies~$|\hJ| \le \delta_2$ and~$C_1 \hJ \subset I$, let~$n \ge 1$ be an integer, and let~$\hW$ be a pull-back of~$\hJ$ by~$f^n$.
Put~$m_0 \= n$ and~$\hW_0 \= \hJ$ and define inductively an integer $k \ge 0$ and integers
$$  m_0 > m_1 > \cdots > m_k \ge 0, $$
such that for each~$t$ in~$\{1, \ldots, k \}$ the pull-back~$\hW_t$ of~$\hJ$ by~$f^{n - m_t}$ containing~$f^{m_t}(\hW)$ is contained in~$B(\Crit(f), \kappa)$.
Note that by our choice of~$\delta_2$ this last property implies that~$C_1 \hW_t \subset I$.
Recalling that~$m_0 = n$, let~$t \ge 0$ be an integer such that~$m_t$ is already defined.
If~$m_t = 0$, or if the pull-back of~$C_1 \hW_t$ by~$f^{m_t}$ containing~$\hW$ is diffeomorphic, then put~$k = t$ and stop.
Otherwise, define~$m_{t + 1}'$ as the largest integer~$m$ in~$\{0, \ldots, m_t - 1 \}$ such that the pull-back~$\hW_{t + 1}'$ of~$C_1 \hW_t$ by~$f^{m_t - m}$ containing~$f^m(\hW)$ is not diffeomorphic.
In view of Lemma~\ref{l:pull-backs}(1), it follows that~$\hW_{t + 1}'$ contains a unique critical point and that this critical point is in~$J(f)$.
Moreover, $\hW_{t + 1}'$ is a child of~$C_1 \hW_t$.
Define~$m_{t + 1}$ as the smallest integer~$m$ in~$\{ 0, \ldots, m_{t + 1}' \}$ such that the pull-back~$W_*$ of~$C_1 \hW_t$ by~$f^{m_t - m}$ containing~$f^m(\hW)$ is a child of~$C_1 \hW_t$.
Clearly, $\hW_{t + 1} \subset W_* \subset B(\Crit(f), \kappa)$.

Note that if~$k = 0$, then the pull-back of~$C_j \hJ$ by~$f^n$ containing~$\hW$ is diffeomorphic; in particular~$f^n \colon \hW \to \hJ$ is diffeomorphic.
On the other hand, note that for every~$t$ in~$\{ 1, \ldots, k - 1 \}$ the unique critical point in~$\hW_{t + 1}'$ is exposed.
So, by definition of~$L$ we have
$$ m_t - m_{t + 1} \ge m_t - m_{t + 1}' \ge L. $$

To prove item~$1$ of the lemma, observe that if~$k = 0$, then~$f^n \colon \hW \to \hJ$ is a diffeomorphism and the desired assertion is trivially true.
Suppose~$k \ge 1$ and let~$J$ be an interval contained in~$\hJ$.
It follows from the definitions that for every~$t$ in~$\{1, \ldots, k \}$ the map~$f^{m_{t - 1} - m_t}$ has at most one critical point in~$f^{m_t}(\hW)$.
Furthermore, an induction argument in~$t$ shows that there are at most~$2^t$ pull-backs of~$J$ by~$f^{n - m_t}$ contained in the pull-back of~$\hJ$ containing~$f^{m_t}(\hW)$.
Since
$$ 2^k \le 2 \eta^{(k - 1)L} \le 2 \eta^{m_1 - m_k} \le 2 \eta^n, $$
the last assertion with~$t = k$ proves item~$1$ of the lemma in the case where~$m_k = 0$.
If~$m_k \ge 1$, then it follows from the definitions that the pull-back of~$C_1 \hW_k$ by~$f^{m_k}$ containing~$\hW$ is diffeomorphic.
So the number of pull-backs of~$J$ by~$f^n$ contained in~$\hW$ is also bounded from above by~$2 \eta^n$.
This completes the proof of item~$1$ of the lemma.

To prove item~$2$, suppose first~$k = 0$.
Then the pull-back of~$C_1 \hJ$ by~$f^n$ containing~$\hW$ is diffeomorphic and the desired inequality follows from the Koebe principle (Lemma~\ref{l:Koebe principle}) with~$12 \hellmax \eta^n$ replaced by~$2$. 
Suppose~$k \ge 1$ and observe that by Lemma~\ref{l:child} for each~$t$ in~$\{1, \ldots, k \}$ we have
$$ |\hW_t|
\le
6 \hellmax |\hW_{t - 1}| \max \left\{ |Df^{m_{t - 1} - m_{t}}(a)| : a \in \partial \hW_{t} \right\}^{-1}. $$
By an induction argument we obtain,
$$ |\hW_{k}|
\le
(6 \hellmax)^{k} |\hJ| \max \left\{ |Df^{n - m_{k}}(a')| : a' \in \partial \hW_{k} \right\}^{-1}. $$
Using
$$(6 \hellmax)^{k - 1} < \eta^{(k - 1)L} \le \eta^{m_1 - m_{k}} \le \eta^n, $$
we obtain
$$ |\hW_k| \le 6 \hellmax \eta^n \max \left\{ |Df^{n - m_{k}}(a) : a \in \partial \hW_k \right\}^{-1}. $$
This proves item~$2$ of the lemma in the case where~$m_k = 0$.
If~$m_k \ge 1$, then the pull-back of~$C_1 \hW_k$ by~$f^{m_k}$ containing~$\hW$ is diffeomorphic and by the Koebe principle (Lemma~\ref{l:Koebe principle}) we obtain
\begin{multline*}
|\hW|
\le
2 |\hW_k| \max \left\{ |Df^{m_k}(a)| : a \in \partial \hW \right\}^{-1}
\\ \le
12 \hellmax |\hJ| \max \left\{ |Df^n(a)| : a \in \partial \hW \right\}^{-1}.
\end{multline*}
This completes the proof of item~$2$ and of the lemma.
\end{proof}

The following lemma is more general than what we need for the proof of Proposition~\ref{p:CE2 implies ESC}.
It is used again in the proof of the Main Theorem in the next section.

\begin{lemm}
\label{l:periodic bound}
Let~$f \colon I \to I$ be an interval map in~$\sA$ that is topologically exact on~$J(f)$ and put
$$ \chiper^0(f)
\=
\inf \left\{ \chi_p(f) : p \text{ periodic point of~$f$ in~$J(f)$} \right\}. $$
Then for every interval~$J$ contained in~$I$ that intersects~$J(f)$ we have
\begin{equation}
\label{e:ESC and periodics}
\liminf_{n \to + \infty} \frac{1}{n} \ln \max \left\{ |W| : W \text{ connected component of~$f^{-n}(J)$} \right\}
\ge
- \chiper^0(f)
\end{equation}
and for every point~$x_0$ of~$J(f)$ we have
\begin{equation}
\label{e:CE2 and periodics}
\limsup_{n \to + \infty} \frac{1}{n} \ln \min \left\{ |Df^n(x)| : x \in f^{-n}(x_0) \right\}
\le
\chiper^0(f).
\end{equation}
\end{lemm}
\begin{proof}
Let~$\ell \ge 1$ be an integer and let~$p$ be a periodic point of~$f$ of period~$\ell$ in~$J(f)$.

Suppose first~$p$ is hyperbolic repelling.
Then there is~$\delta > 0$ and a uniformly contracting inverse branch~$\phi$ of~$f^{\ell}$ that is defined on~$B(p, \delta)$ and fixes~$p$.
It follows that $\phi(\overline{B(p, \delta)}) \subset B(p, \delta)$ and that there is~$K > 1$ such that for every integer~$k \ge 1$ the distortion of~$\phi^k$ on~$B(p, \delta)$ is bounded by~$K$.
On the other hand, the hypothesis that~$f$ is topologically exact on~$J(f)$ implies that there is an integer~$m \ge 1$ such that the intersection of~$f^{-m}(J)$ and~$B(p, \delta)$ contains an interval~$J'$ and such that there is a point~$x_0'$ in~$f^{-m}(x_0)$ contained in~$B(p, \delta)$.
Then we have
\begin{multline}
\label{e:inverse ESC}
\liminf_{n \to + \infty} \frac{1}{n} \ln \max \left\{ |W| : W \text{ connected component of~$f^{-n}(J)$} \right\}
\\ \ge
\liminf_{k \to + \infty} \frac{1}{k\ell} \ln |\phi^k(J')|
=
- \chi_p(f)
\end{multline}
and
\begin{multline}
\label{e:inverse CE2}
\limsup_{n \to + \infty} \frac{1}{n} \ln \min \left\{ |Df^n(x)| : x \in f^{-n}(x_0) \right\}
\\ \le
- \lim_{k \to + \infty} \frac{1}{k\ell} \ln |D\phi^k(x_0')|
=
\chi_p(f).
\end{multline}
Since~$p$ is an arbitrary hyperbolic repelling periodic point, this proves~\eqref{e:ESC and periodics} and~\eqref{e:CE2 and periodics}.

It remains to consider the case where~$p$ is not hyperbolic repelling, so that~$Df^{2\ell}(p) = 1$.
Without loss of generality we assume that for every~$\delta > 0$ the interval~$(p, p + \delta)$ intersects~$J(f)$.
Let~$\eta > 1$ be given and let~$\delta > 0$ be sufficiently small so there is an inverse branch~$\phi$ of~$f^{2 \ell}$ that is defined on~$B(p, \delta)$, that fixes~$p$, and that is strictly increasing on~$(p, p + \delta)$.
Reducing~$\delta$ if necessary we assume we have~$|Df| < \eta$ on~$B(p, \delta)$.
As in the previous case there is an integer~$m \ge 1$ such that the intersection of~$f^{-m}(J)$ and~$(p, p + \delta)$ contains an interval~$J'$ and such that there is a point~$x_0'$ in~$f^{-m}(x_0)$ contained in~$(p, p + \delta)$.
Then we have~\eqref{e:inverse ESC} and~\eqref{e:inverse CE2} with~$\chi_p(f)$ replaced by~$\varepsilon$.
Since~$\varepsilon > 0$ is arbitrary, these inequalities hold with~$\chi_p(f) = 0$.
The proof of the lemma is thus completed.
\end{proof}

\begin{proof}[Proof of Proposition~\ref{p:CE2 implies ESC}]
  By Lemma~\ref{l:periodic bound} all the periodic points of~$f$ in~$J(f)$ are hyperbolic repelling.
  It is enough to show that for every~$\hlambda_0$ in~$(\lambda_0, \lambda)$ there is a constant~$C_0 > 0$ such that the proposition holds with the right hand side of~\eqref{e:weak ESC} replaced by~$C_0 \hlambda_0^{-n}$.

Let~$\tI$ be equal to~$I$ if~$J(f) = I$.
Otherwise, for each periodic point~$y$ in the boundary of a Fatou component~$U$ of~$f$, let~$y'$ be a point in~$U$, let~$U_y$ be the open interval bounded by~$y$ and~$y'$, and put
$$ \tI \= I \setminus \bigcup_y U_y, $$
where the union runs through all the periodic points of in the boundary of a Fatou component of~$f$.
In all the cases~$\tI$ is a finite union of closed intervals.
In part~$1$ below we show that for every~$y$ in~$J(f)$ there is a constant~$C_y > 0$ and an interval~$J_y$ contained in~$\tI$ that is a neighborhood of~$y$ in~$\tI$ and such that for every integer~$n \ge 1$ and every pull-back~$W$ of~$J_y$ by~$f^n$ we have
$$ |W| \le C_y \hlambda_0^{-n}. $$
Since~$J(f)$ is compact, this implies the proposition, except in the case where~$J(f)$ is an interval having a boundary point in the interior of~$I$ that is a periodic point of~$f$.
This last case is treated in part~$2$.

Let~$\hI$ be a compact interval containing~$I$ in its interior and let~$\whf \colon \hI \to \hI$ be an extension of~$f$ in~$\sA$ that is boundary anchored, such that all the critical points of~$\whf$ are contained in~$I$, and such that~$\bigcup_{n = 0}^{+ \infty} \whf^{-n}(I)$ contains an interval whose closure contains~$J(\whf)$.
Note in particular that~$\whf$ is essentially topologically exact on~$J(\whf)$.
Without loss of generality we assume that all the periodic points of~$\whf$ in~$J(\whf)$ are hyperbolic repelling.
Put~$\eta \= (\lambda / \hlambda_0)^{1/2}$ and let~$\delta_* > 0$ be the constant~$\delta(\eta)$ given by Lemma~\ref{l:maximum principle} with~$f$ replaced by~$\whf$.
Moreover, let~$C_1 > 1$ be the constant given by Lemma~\ref{l:child}.
Reducing~$\delta_*$ if necessary we assume
$$ \delta_* < C_1^{-1} \dist(I, \partial \hI). $$
Note that this implies that for every interval~$J$ intersecting~$I$ and satisfying~$|J| \le \delta_*$, we have~$C_1 J \subset \hI$.

\partn{1}
Suppose first~$y$ is not a boundary point of a Fatou component of~$f$ of length greater than or equal to~$\delta_*/2$.
Since~$f$ is topologically exact on~$J(f)$, we can find an integer~$n_0 \ge 1$ and points~$x$ and~$x'$ in~$f^{-n_0}(x_0)$ such that
$$ x < y < x'
\text{ and }
|x - x'| < \delta_*. $$
Then the desired assertion follows with
$$ J_y = (x, x')
\text{ and }
C_y = 12 \hellmax C^{-1} \delta_*, $$
by Lemma~\ref{l:maximum principle}(2) with~$f$ replaced by~$\whf$ and with~$\hJ = (x, x')$.

Suppose~$y$ is a boundary point of a Fatou component of~$f$ and that~$y$ is not periodic.
Then there is an integer~$N \ge 1$ such that every point in~$f^{-N}(y)$ is either not in the boundary of a Fatou component or in the boundary of a Fatou component of length strictly smaller than~$\delta_*/2$.
Then the desired assertion follows from the previous case.

It remains to consider the case where~$y$ is a periodic point in the boundary of a Fatou component of length greater than or equal to~$\delta_*/2$.
Let~$\ell \ge 1$ be the period of~$y$ and let~$\delta$ in~$(0, \delta_*/2)$ be sufficiently small so that there is an inverse~$\phi$ of~$\whf^\ell$ defined on~$B(y, \delta)$, fixing~$y$ and such that~$\phi(\overline{B(y, \delta)}) \subset B(y, \delta)$.
Since~$\delta < \delta_*/2$ and~$y$ is a boundary point of a Fatou component of~$f$ of length greater than or equal to~$\delta_*/2$, it follows that~$\phi$ is strictly increasing.
Let~$n_0 \ge 1$ be a sufficiently large integer so that~$f^{-n_0}(x_0)$ intersects $B(y, \delta)$ and let~$y_0$ be a point of~$f^{-n_0}(x_0)$ in~$B(y, \delta)$.
For each integer~$j \ge 1$ put~$y_j \= \phi^j(y_0)$ and let~$K_{j - 1}$ be the closed interval bounded by~$y_{j - 1}$ and~$y_j$.
Note that the intervals~$(K_j)_{j = 0}^{+ \infty}$ have pairwise disjoint interiors and that the closure of their union is equal to the closed interval~$J_y$ bounded by~$y$ and~$y_0$.
Clearly~$J_y$ is a neighborhood of~$y$ in~$\tI$.
On the other hand, for each integer~$j \ge 1$ the interval~$K_j$ is equal to~$\phi^j(K_0)$ and it is a pull-back of~$K_0$ by~$\whf^{\ell j}$.
So, Lemma~\ref{l:maximum principle}(2) with~$\hJ = K_0$, with~$f$ replaced by~$\whf$, and with~$n$ replaced by~$n + \ell j$, shows that for every pull-back~$W$ of~$K_j$ by~$\whf^n$ we have
\begin{equation*}
  \begin{split}
|W|
& \le
12 \hellmax \eta^{n + j\ell } |K_0| \max \left\{ |D\whf^{n + j\ell }(a)| : a \in \partial W \right\}^{-1}
\\ & \le
12 \hellmax \eta^{n + j\ell } \delta_* C^{-1} \lambda^{-(n + j\ell  + n_0)} \min \left\{ |D\whf^{n_0}(y_0)|^{-1}, |D\whf^{n_0 + \ell }(y_1)|^{-1} \right\}.
  \end{split}
\end{equation*}
On the other hand, by Lemma~\ref{l:maximum principle}(1) with~$f$ replaced by~$\whf$ and with~$\hJ = J_y$ and~$J = K_j$, every pull-back~$\hW$ of~$J_y$ by~$f^n$ contains at most~$2 \eta^n$ pull-backs of~$K_j$ by~$f^n$.
So, letting
$$ C'
\=
12 \hellmax \delta_* C^{-1} \lambda^{- n_0} \min \left\{ |D\whf^{n_0}(y_0)|^{-1}, |D\whf^{n_0 + \ell }(y_1)|^{-1} \right\} $$
and using the definition of~$\eta$ we obtain
\begin{equation*}
|\hW \cap \whf^{-n}(K_j)|
\le
2 \eta^n C' \eta^{n + j\ell } \lambda^{-(n + j\ell )}
\le
2 C' \hlambda_0^{- (n + j\ell )}.
\end{equation*}
Since~$J_y$ is the closure of~$\bigcup_{j \ge 0} K_j$, summing over~$j$ we get
$$ |\hW|
\le
2C' \sum_{j = 0}^{+ \infty} \hlambda_0^{-(n + j\ell )}
=
2C' (1 - \hlambda_0^{-\ell })^{-1} \hlambda_0^{-n}. $$
This proves the desired assertion with~$C_y = 2 C' (1 - \hlambda_0^{-\ell })^{-1}$.

\partn{2}
Suppose that~$J(f)$ is an interval having a boundary point~$y$ in the interior of~$I$ that is a periodic point of~$f$.
In view of part~$1$, it is enough to show that for each such point~$y$ there are~$\delta > 0$ and~$C > 0$ such that for every integer~$n \ge 1$ and every pull-back~$W$ of~$B(y, \delta)$ by~$f^n$, we have~$|W| \le C \hlambda_0^{-n}$.
By part~$1$ there are~$\delta > 0$ and~$C > 0$ such that this property holds with~$B(y, \delta)$ replaced by the interval~$J \= B(y, \delta) \cap J(f)$.

Let~$\cO$ be the forward orbit of~$y$.
Note that~$\cO \subset \partial I$, that the set~$\cO' \= f^{-1}(\cO) \cap \partial J(f)$ is forward invariant, and that~$f^{-1}(\cO') \setminus \cO'$ is contained in the interior of~$J(f)$.
Reducing~$\delta$ if necessary assume that each pull-back of~$B(y, \delta)$ by~$f$ or by~$f^2$ that is disjoint from~$\cO'$ is contained in~$J(f)$.
It follows that for every integer~$n \ge 1$, each pull-back~$W$ of~$B(y, \delta)$ by~$f^n$ that is disjoint from~$\cO'$ is contained in~$J(f)$ and therefore coincides with a pull-back of~$J$ by~$f^n$.
By our choice of~$\delta$, in this case we have~$|W| \le C \hlambda_0^{-n}$.
It remains to consider those pull-backs~$W$ of~$B(y, \delta)$ that intersect~$\cO'$.
Since by Lemma~\ref{l:periodic bound} the periodic point~$y$ satisfies~$\chi_y(f) \ge \ln \lambda$, reducing~$\delta$ if necessary we can assume that for every integer~$n \ge 1$ and every pull-back~$W$ of~$B(y, \delta)$ by~$f^n$ that intersects~$\cO'$, we have~$|W| \le C\hlambda_0^{-n}$.
This completes the proof of the proposition.
\end{proof}

\section{Quantifying asymptotic expansion}
\label{s:quantified expansion}
The purpose of this section is to prove the following strengthened version of the Main Theorem.
Given a compact space~$X$ and a continuous map $T \colon X \to X$, we denote by~$\sM(X, T)$ the space of Borel probability measures on~$X$ that are invariant by~$T$.

\begin{generic}[Main Theorem']
For an interval map~$f$ in~$\sA$, the number
$$ \chiperf
\=
\inf \left\{ \chi_p(f) : p \text{ hyperbolic repelling periodic point of~$f$ in~$J(f)$} \right\} $$
is equal to
$$ \chiinff
\=
\left\{ \chi_\nu(f) : \nu \in \sM(J(f), f) \right\}. $$
If in addition~$f$ is topologically exact on~$J(f)$, then there is~$\delta' > 0$ such that the following properties hold.
Let~$J$ be an interval contained in~$I$ that intersects~$J(f)$ and satisfies~$|J| \le \delta'$.
In the case where~$\chiinff > 0$ and where~$J(f)$ is not an interval,
assume in addition that~$J$ is not a neighborhood of a periodic point in the boundary of a Fatou component of~$f$.
Then:
\begin{enumerate}
\item[1.]
For every~$\chi < \chiinff$ there is a constant~$C > 0$ independent of~$J$, such that for every integer~$n \ge 1$ and every pull-back~$W$ of~$J$ by~$f^n$, we have $|W| \le C \exp(- n \chi)$.
\item[2.]
We have
$$ \lim_{n \to + \infty} \frac{1}{n} \ln \max \left\{ |W| : W \text{ connected component of~$f^{-n}(J)$} \right\}
=
- \chiinff. $$
\end{enumerate}
Finally, for each point~$x_0$ in~$J(f)$ we have
$$ \limsup_{n \to + \infty} \frac{1}{n} \ln \min \left\{ |Df^n(x)| : x \in f^{-n}(x_0) \right\}
\le
\chiinff, $$
and there is a subset~$E$ of~$J(f)$ of zero Hausdorff dimension such that for each point~$x_0$ in~$J(f) \setminus E$  the $\limsup$ above is a limit and the inequality an equality.
\end{generic}

\begin{rema}
\label{r:disconnected ESC}
In the case where~$\chiinff > 0$ and where~$J(f)$ is not an interval, there is an example showing that the hypothesis in the Main Theorem' that~$J$ is not a neighborhood of a periodic point in the boundary of a Fatou component, is necessary, see~\cite[Proposition~A]{Riv1206}.
However, a qualitative result holds when this hypothesis is not satisfied, see~\cite[Theorem~B]{Riv1206}.
\end{rema}

The proof of the Main Theorem' is given below, after the following lemmas from~\cite{PrzRivSmi03}.

When~$f$ is a complex rational map the following lemma is a direct consequence of~\cite[Lemma~$3.1$]{PrzRivSmi03}.
Using Fact~\ref{f:positive entropy}, the proof applies without change to the case where~$f$ is a map in~$\sA$.
\begin{lemm}
\label{l:UHP implies CE2}
Let~$f$ be an interval map in~$\sA$ that is topologically exact on~$J(f)$ and such that~$\chiperf > 0$.
Then there is a point~$x_0$ in~$J(f)$ such that
$$ \liminf_{n \to + \infty} \frac{1}{n} \ln \min \left\{ |Df^n(x)| : x \in f^{-n}(x_0) \right\}
\ge
\chiperf. $$
\end{lemm}
In the case where~$f$ is a complex rational map, the following is~\cite[Lemma~$2.1$ and Remark~$2.2$]{PrzRivSmi03}.
The proof applies without change to maps in~$\sA$.
\begin{lemm}
\label{l:exceptional set}
Let~$f \colon I \to I$ be a map in~$\sA$.
Then there are~$\delta_4 > 0$ and a subset~$E$ of~$I$ of zero Hausdorff dimension, such that for every interval~$J$ contained in~$I$ that intersects~$J(f)$ and satisfies~$|J| \le \delta_4$ and every point~$x_0$ in~$J \setminus E$, we have
\begin{multline*}
\liminf_{n \to + \infty} \frac{1}{n} \ln \min \left\{ |Df^n(x)| : x \in f^{-n}(x_0) \right\}
\\ \ge
- \limsup_{n \to + \infty} \frac{1}{n} \ln \max \left\{ |W| : W \text{ connected component of~$f^{-n}(J)$} \right\}.
\end{multline*}
\end{lemm}
\begin{proof}[Proof of the Main Theorem']
To prove
\begin{equation}
  \label{e:main equality}
\chiinff = \chiperf,  
\end{equation}
suppose~$f$ is ``infinitely renormalizable,'' see~\cite{dMevSt93} for the definition and for precisions.
It follows easily from the \emph{a priori} bounds in~\cite{vStVar04} that in this case we have~$\chiinff = \chiperf = 0$.
So, to prove~\eqref{e:main equality} it is enough to consider the case where~$f$ is at most finitely renormalizable.
Then~$f$ can be decomposed into finitely many interval maps, each of which has a renormalization with a topologically exact restriction, see for example~\cite[\S{III}, $4$]{dMevSt93}.
Thus, to prove the Main Theorem' it is enough to consider the case where~$f$ is topologically exact.

In part~$1$ below we prove item~$1$ of the theorem with~$\chiinff$ replaced by~$\chiperf$ and in part~$2$ we prove~$\chiperf = \chiinff$.
We complete the proof of the theorem in part~$3$.

\partn{1}
We prove item~$1$ of the theorem with~$\chiinff$ replaced by~$\chiperf$.
This statement being trivial in the case where~$\chiperf = 0$, we suppose~$\chiperf > 0$.
Combining Lemma~\ref{l:UHP implies CE2} and Proposition~\ref{p:CE2 implies ESC} we obtain that all the periodic points of~$f$ in~$J(f)$ are hyperbolic repelling and that for every~$\chi$ in~$(0, \chiperf)$ there is~$\delta(\chi) > 0$ such that for every interval~$J$ that intersects~$J(f)$, that is disjoint from each periodic Fatou component of~$f$, and that satisfies~$|J| \le \delta(\chi)$, the following property holds: For every integer~$n \ge 1$ and every pull-back~$W$ of~$J$ by~$f^n$ we have
$$ |W| \le \exp( -n \chi). $$
Put $\delta' \= \delta(\chiperf/2)$ and let~$J$ be an interval that intersects~$J(f)$, that is disjoint from the periodic Fatou components of~$f$, and that satisfies~$|J| \le \delta'$.
Given~$\chi$ in~$(\chiperf/2, \chiperf)$, let~$N \ge 1$ be sufficiently large so that~$\exp(-N \chi) \le \delta(\chi)$, let~$n \ge N$ be an integer, and let~$W$ be a pull-back of~$J$ by~$f^n$.
If we denote by~$W'$ the pull-back of~$J$ by~$f^N$ containing~$f^{n - N}(W)$, then we have
$$ |W'| \le \exp(-N \chi) \le \delta(\chi). $$
So the property above applied to~$W'$ instead of~$J$ implies
$$ |W| \le \exp(- (n - N) \chi). $$
This proves item~$1$ of the theorem with~$C = \exp(N \chi)$ and with~$\chiinff$ replaced by~$\chiperf$.

\partn{2}
We prove~$\chiperf = \chiinff$.
To prove~$\chiperf \ge \chiinff$, let~$p$ be a hyperbolic repelling periodic point of~$f$ in~$J(f)$ and let~$\nu$ be the probability measure equidistributed on the orbit of~$p$.
Then~$\nu$ is in~$\sM(J(f), f)$ and~$\chi_{\nu}(f) = \chi_p(f)$, so~$\chi_p(f) \ge \chiinff$.
This proves~$\chiperf \ge \chiinff$.
To prove the reverse inequality we show that for every~$\nu$ in~$\sM(J(f), f)$ we have~$\chi_{\nu}(f) \ge \chiperf$.
By the ergodic decomposition theorem we can assume without loss of generality that~$\nu$ is ergodic.
By~\cite[Theorem~B]{Prz93} or by Proposition~\ref{p:Lyapunov are nonnegative} in Appendix~\ref{s:Lyapunov are nonnegative}, we have~$\chi_{\nu}(f) \ge 0$.
We show that for every~$\varepsilon > 0$ there is a point~$x$ in~$J(f)$ such that for every sufficiently large integer~$n \ge 1$ we have
\begin{equation}
\label{e:exponentially small disk}
f^n(B(x, \exp(- (\chi_{\nu}(f) + 2\varepsilon) n )))
\subset
B(f^n(x), \exp(- \varepsilon n)).
\end{equation}
Using this estimate with a sufficiently large~$n$ and combining it with part~$1$ we obtain~$\chi_\nu(f) + 2 \varepsilon \ge \chiperf$.
Since~$\nu$ and~$\varepsilon$ are arbitrary, this proves~$\chiinff \ge \chiperf$, as wanted. 
To prove~\eqref{e:exponentially small disk}, note that by Birkhoff's ergodic theorem there is a point~$x_0$ in~$J(f)$ and an integer~$n_0 \ge 1$ such that for every~$n \ge n_0$ we have
\begin{equation}
\label{e:finite time Lyapunov}
\exp \left( \left( \chi_{\nu}(f) - \tfrac{1}{3} \varepsilon \right) n \right)
\le
|Df^n(x_0)|
\le
\exp \left( \left( \chi_{\nu}(f) + \tfrac{1}{3} \varepsilon \right) n \right).
\end{equation}
On the other hand, since the critical points of~$f$ are non-flat, there are constants~$C_0 > 0$ and~$\alpha > 0$ such that for every~$x$ in~$I$ we have
\begin{displaymath}
  |Df(x)|
  \le
  C_0 \dist(x, \Crit(f))^{\alpha}.  
\end{displaymath}
Put~$\varepsilon' \= \frac{\varepsilon}{\alpha}$.
Using the previous inequality with~$x = f^n(x_0)$, combined with
$$ Df^{n + 1}(x_0) = Df(f^n(x_0)) \cdot Df^n(x_0), $$
with~\eqref{e:finite time Lyapunov} and with~\eqref{e:finite time Lyapunov} with~$n$ replaced by~$n + 1$, we obtain that for every~$n \ge n_0$ we have
$$ \dist(f^n(x), \Crit(f))
\ge
\left( C_0^{-1} \exp(\chi_\nu(f)) \right)^{\frac{1}{\alpha}} \exp \left( - \tfrac{2}{3} \varepsilon' (n + 1) \right). $$
This implies that there is an integer~$n_1 \ge n_0$ such that for every~$n \ge n_1$ the distortion of~$f$ on~$B(f^n(x_0), \exp(- \varepsilon' n))$ is bounded by~$\exp \left( \tfrac{1}{3} \varepsilon' \right)$.
Let~$n_2 \ge n_1$ be sufficiently large so that the distortion of~$f^{n_1}$ on~$B(x_0, \exp(- (\chi_\nu(f) + \varepsilon')n_2))$ is bounded by~$\exp \left( \tfrac{1}{3} \varepsilon' n_1 \right)$.
Then for every~$n \ge n_2$ we have,
\begin{multline}
\label{e:exponentially small inclusion}
f^{n_1}(B(x_0, \exp(- (\chi_\nu(f) + 2\varepsilon')n)))
\\ \subset
B \left( f^{n_1}(x_0), \exp \left(- (\chi_\nu(f) + 2\varepsilon') n + \tfrac{1}{3} \varepsilon' n_1 \right) |Df^{n_1}(x_0)| \right).
\end{multline}
Fix~$n \ge n_2$.
We prove by induction that for every~$j$ in~$\{n_1, \ldots, n \}$ the inclusion above holds with~$n_1$ replaced by~$j$.
The desired assertion is obtained from this with~$j = n$, combined with~\eqref{e:finite time Lyapunov}.
Noting that the case~$j = n_1$ is given by~\eqref{e:exponentially small inclusion} itself, let~$j$ in~$\{n_1, \ldots, n - 1 \}$ be given and suppose~\eqref{e:exponentially small inclusion} holds with~$n_1$ replaced by~$j$.
Then~\eqref{e:exponentially small inclusion} with~$n_1$ replaced by~$j + 1$ is obtained by using that the right hand side of~\eqref{e:exponentially small inclusion} with~$n_1$ replaced by~$j$ is contained in~$B(f^j(x_0), \exp(- \varepsilon' n))$, combined with the fact that the distortion of~$f$ on this last set is bounded by~$\exp \left( \tfrac{1}{3} \varepsilon' \right)$.
This completes the proof of the induction step, and hence that~$\chi_{\nu}(f) \ge \chiperf$ and~$\chiinff = \chiperf$.

\partn{3}
So far we have shown item~$1$ of the theorem and the equality~$\chiinff = \chiperf$.
Let~$\chiper^0(f)$ be as in the statement of Lemma~\ref{l:periodic bound}.
Clearly,
$$ \chiinff \le \chiper^0(f) \le \chiperf $$
(\emph{cf.}, first part of part~$2$), so~$\chiper^0(f) = \chiinff$.
Thus, inequality~\eqref{e:ESC and periodics} of Lemma~\ref{l:periodic bound} and item~$1$ of the theorem imply item~$2$ of the theorem.
In turn, item~$2$ of the theorem together with~\eqref{e:CE2 and periodics} of Lemma~\ref{l:periodic bound} and with Lemma~\ref{l:exceptional set} imply the last assertion of the theorem.
The proof of the theorem is thus complete.
\end{proof}

\section{Conjugacy to a piecewise affine map}
\label{s:conjugacy to affine}
In this section we show that a conjugacy between~$2$ Lipschitz continuous multimodal maps that satisfy the Exponential Shrinking of Components condition\footnote{The Exponential Shrinking of Components condition is defined in~\S\ref{ss:NUH} for non-degenerate smooth interval maps. In this section we apply this definition to multimodal maps.} is bi-H{\"o}lder continuous (Proposition~\ref{p:conjugacy to affine}).
Combined with Lemma~\ref{l:affine model satisfies ESC} below, this proves implication~$5 \Rightarrow 4$ of Corollary~\ref{c:equivalences}.

A multimodal map~$f$ is \emph{expanding}, if there is~$\lambda > 1$ so that for every~$x$ and~$x'$ contained in an interval on which~$f$ is monotonous, we have
$$ |f(x) - f(x')|
\ge
\lambda |x - x'|. $$
In this case we say~$\lambda$ is an \emph{expansion constant of~$f$}.
\begin{lemm}
\label{l:affine model satisfies ESC}
Every expanding multimodal map satisfies the Exponential Shrinking of Components condition.
\end{lemm}
In this section, a turning point~$c$ of a multimodal map~$f$ is~\emph{exposed} if for every integer~$n \ge 1$ the point~$f^n(c)$ is not a turning point of~$f$.
\begin{proof}
Let~$f \colon I \to I$ be an expanding multimodal map and let~$\lambda > 1$ be an expansion constant of~$f$.
Let~$L \ge 1$ be a sufficiently large integer so that~$\lambda^L > 2$ and let~$\delta_{\dag} > 0$ be sufficiently small so that for every exposed turning point~$c$ of~$f$ and every~$j$ in~$\{1, \ldots, L \}$ the set~$f^j(B(c, \delta_{\dag}))$ does not contain a turning point of~$f$.
Let~$\delta_* > 0$ be sufficiently small so that for every interval~$J$ contained in~$I$ that satisfies~$|J| \le \delta_*$ and every connected component~$W$ of~$f^{-1}(J)$ we have~$|W| \le \delta_{\dag}$.

We prove by induction on~$n \ge 0$ that for every interval~$J$ contained in~$I$ that satisfies~$|J| \le \delta_*/2$, every~$j$ in~$\{1, \ldots, n \}$, and every pull-back~$W$ of~$J$ by~$f^j$ we have
$$ |W| \le \left( 2^{\frac{1}{L}} \lambda^{-1} \right)^j \delta_*. $$
This implies that~$f$ satisfies the Exponential Shrinking of Components condition.
The case~$n = 0$ being trivial, suppose that for some~$n \ge 1$ this assertion holds with~$n$ replaced by each element of~$\{0, \ldots, n - 1 \}$.
Let~$J$ be an interval contained in~$I$ that satisfies~$|J| \le \delta_*/2$ and let~$W$ be a pull-back of~$J$ by~$f^n$.
The induction hypothesis implies for every~$j$ in~$\{1, \ldots, n - 1 \}$ we have~$|f^j(W)| \le \delta_*$.
Using the hypothesis~$|J| \le \delta_*/2$ and the definition of~$\delta_*$, we conclude that for every~$i$ in~$\{0, \ldots, n - 1 \}$ we have~$|f^i(W)| \le \delta_{\dag}$.
Using the definition of~$\delta_{\dag}$, this implies that the number of those~$i$ in~$\{0, \ldots, n - 1 \}$ such that~$f^i(W)$ contains a turning point of~$f$ in its interior is at most~$\frac{n}{L} + 1$.
It thus follows that~$W$ can be partitioned into at most~$2^{\frac{n}{L} + 1}$ intervals on each of which~$f^n$ is injective.
Using that~$\lambda$ is an expansion constant of~$f$, we obtain
$$ |W|
\le
2^{\frac{n}{L} + 1} \lambda^{-n} |J|
\le
2^{\frac{n}{L}} \lambda^{-n} \delta_*. $$
This completes the proof of the induction hypothesis and of the lemma.
\end{proof}

\begin{prop}
\label{p:conjugacy to affine}
Let~$f \colon I \to I$ be a Lipschitz continuous multimodal map and~${\wtf \colon \tI \to \tI}$ a multimodal map satisfying the Exponential Shrinking of Components condition.
If~$h \colon I \to \tI$ is a homeomorphism conjugating~$f$ to~$\wtf$, then~$h$ is H{\"o}lder continuous.
\end{prop}
We deduce this proposition as an easy consequence of the following lemma.
\begin{lemm}
\label{l:injective time}
Let~$f \colon I \to I$ be a multimodal map satisfying the Exponential Shrinking of Components condition with constant~$\lambda > 1$.
Then for every~$A > (\ln \lambda)^{-1}$ there is a constant~$\delta_5 > 0$ such that for every interval~$J$ contained in~$I$ the following property holds: There is an integer~$m \ge 0$ that satisfies~$m \le \max \{ - A \ln |J|, 0 \}$ and an interval~$J_0$ contained in~$J$, such that~$f^m$ is injective on~$J_0$ and~$|f^m(J_0)| \ge \delta_5$.
\end{lemm}
\begin{proof}
Put~$\chi \= \ln \lambda$ and let~$L$ be an integer satisfying $L > (A \chi - 1)^{-1} A \ln 2$.
Let~$\delta_{\dag} > 0$ be sufficiently small so that for every exposed turning point~$c$ of~$f$ and for every~$j$ in~$\{1, \ldots, L \}$, the set~$f^j(B(c, \delta_{\dag}))$ does not contain a turning point of~$f$.
Let~$\delta_{\Exp} > 0$ be the constant~$\delta$ given by the Exponential Shrinking of Components condition, see~\S\ref{ss:NUH}.
Reducing~$\delta_{\Exp}$ if necessary we assume that for every interval~$J$ contained in~$I$ that satisfies~$|J| \le \delta_{\Exp}$, every integer~$n \ge 1$, and every pull-back~$W$ of~$J$ by~$f^n$ we have~$|W| \le \delta_{\dag}$.
Let~$\delta_{\Exp}^* > 0$ be such that for every interval~$J$ contained in~$I$ that satisfies~$|J| \ge \delta_{\Exp}$ and for every connected component~$W$ of~$f^{-1}(J)$ we have~$|W| \ge \delta_{\Exp}^*$.
Reducing~$\delta_{\Exp}^*$ if necessary we assume~$\delta_{\Exp}^* \le \delta_{\Exp}$.
Observing that~$1 + A \frac{\ln 2}{L} < \chi A$, it follows that there is~$n_0 \ge 1$ such that for every integer~$n \ge n_0$ we have,
\begin{equation}
\label{e:starting time}
- A \ln \frac{\delta_{\Exp}^*}{2} + \left(1 + A \frac{\ln 2}{L} \right) n
\le
\chi A n.
\end{equation}

In part~$1$ below we show that every interval contains an interval that is mapped bijectively by an iterate of~$f$ onto a relatively large interval.
In part~$2$ we use this fact to prove the lemma by induction.

\partn{1}
We prove that for every integer~$n \ge 1$ and every interval~$J$ contained in~$I$ that satisfies~$|J| \ge \exp(- (n + 1) \chi)$, there is~$m$ in~$\{ 0, \ldots, n \}$ and an interval~$J_0$ contained in~$J$ such that~$f^m$ is injective on~$J_0$ and
$$ |f^m(J_0)| \ge \frac{\delta_{\Exp}^*}{2} 2^{- \frac{m}{L}}. $$
If~$|J| \ge \delta_{\Exp}$, then the assertion follows with~$J_0 = J$ and~$m = 0$ from our assumption that~$\delta_{\Exp} \ge \delta_{\Exp}^*$.
Assume~$|J| \le \delta_{\Exp}$ and note that by the Exponential Shrinking of Components condition, for every integer~$m \ge n + 1$ we have~$|f^m(J)| > \delta_{\Exp}$.
So there is a largest integer~$m \ge 0$ such that~$|f^m(J)| \le \delta_{\Exp}$ and~$m$ satisfies~$m \le n$.
By definition of~$\delta_{\Exp}^*$ we have~$|f^m(J)| \ge \delta_{\Exp}^*$.
On the other hand, by our choice of~$\delta_{\Exp}$, for every~$j$ in $\{0, \ldots, m - 1 \}$ we have~$|f^j(J)| \le \delta_{\dag}$.
From the definition of~$\delta_{\dag}$ it follows that the number of those~$j$ in~$\{0, \ldots, m - 1 \}$ such that~$f^j(J)$ contains a turning point in its interior is bounded by~$\frac{m}{L} + 1$.
This implies that~$J$ can be partitioned into at most~$2^{\frac{m}{L} + 1}$ intervals on which~$f^m$ is injective.
So, if we denote by~$J_0$ an interval~$J'$ in this partition for which~$|f^m(J')|$ is maximal, then we have
\begin{equation}
  \label{e:relatively big injective}
|f^m(J_0)|
\ge
\frac{|f^m(J)|}{2^{\frac{m}{L} + 1}}
\ge
\frac{\delta_{\Exp}^*}{2} 2^{- \frac{m}{L}}.   
\end{equation}

\partn{2}
Put~$\delta_5 \= \frac{\delta_{\Exp}^*}{2} 2^{- \frac{n_0}{L}}$.
We prove by induction that for every integer~$n \ge 1$ the lemma holds for every interval~$J$ that satisfies~$|J| \ge \exp(- (n + 1) \chi)$.
Part~$1$ implies that this holds for every integer~$n \ge 0$ satisfying~$n \le n_0$.
Let~$n \ge n_0$ be an integer for which the lemma holds for every interval~$J$ that satisfies~$|J| \ge \exp(-n \chi)$.
To prove the inductive step, let~$J$ be a given interval contained in~$I$ that satisfies
$$ \exp(-(n + 1) \chi) \le |J| \le \exp(- n \chi). $$
Let~$m$ be the integer in~$\{0, \ldots, n \}$ and~$J_0$ the interval contained in~$J$ given by part~$1$.
So~$f^m$ is injective on~$J_0$ and
$$ |f^m(J_0)|
\ge
\frac{\delta_{\Exp}^*}{2} 2^{- \frac{m}{L}}
\ge
\frac{\delta_{\Exp}^*}{2} 2^{- \frac{n}{L}}. $$
Together with~\eqref{e:starting time} this implies $|f^m(J_0)| \ge \exp(- n \chi)$, so we can apply the induction hypothesis with~$J$ replaced by~$f^m(J_0)$.
Therefore there is an interval~$J_0'$ contained in~$f^m(J_0)$ and an integer~$m' \ge 0$ satisfying~$m' \le \max \{ - A \ln |f^m(J_0)|, 0 \}$, such that~$f^{m'}$ is injective on~$J_0'$ and~$|f^{m'}(J_0')| \ge \delta_5$.
If~$m' = 0$, then $|f^m(J_0)| \ge |J_0'| \ge \delta_5$.
Together with
$$ m \le n \le - \chi^{-1} \ln |J| < - A \ln |J|, $$
this completes the proof of the induction step in the case where~$m' = 0$.
Suppose~$m' \ge 1$ and let~$\tJ_0$ be the connected component of~$f^{-m}(J_0')$ contained in~$J_0$, so that~$f^m$ is injective on~$\tJ_0$ and~$f^m(\tJ_0) = J_0'$.
Then~$f^{m + m'}$ is injective on~$\tJ_0$ and~$|f^{m + m'}(\tJ_0)| = |f^{m'}(J_0')| \ge \delta_5$.
On the other hand, we have by~\eqref{e:starting time} and~\eqref{e:relatively big injective}
\begin{multline*}
 m + m'
\le
m - A \ln |f^m(J_0)|
\le
- A \ln \frac{\delta_{\Exp}^*}{2} + \left( 1 + A \frac{\ln 2}{L} \right) m
\\ \le
\chi A n
\le
- A \ln |J|.
\end{multline*}

This completes the proof of the induction step with~$m$ replaced by~$m + m'$ and~$J_0$ replaced by~$\tJ_0$.
The proof of the lemma is thus complete.
\end{proof}
\begin{proof}[Proof of Proposition~\ref{p:conjugacy to affine}]
Denote by~$M$ a Lipschitz constant of~$f$, let~$A$ and~$\delta_5$ be as in Lemma~\ref{l:injective time} with~$f$ replaced by~$\wtf$ and let~$\delta_5^* > 0$ be such that for every interval~$J^*$ contained in~$\tI$ that satisfies~$|J^*| \ge \delta_5$, we have~$|h^{-1}(J^*)| \ge \delta_5^*$.

To prove that~$h$ is H{\"o}lder continuous, let~$J$ be an interval contained in~$I$ and let~$m \ge 0$ be the integer and~$J_0$ the interval given by Lemma~\ref{l:injective time} with~$J$ replaced by~$h(J)$, so that
$$ m \le \max \{- A \ln |h(J)|, 0 \},
J_0 \subset h(J),
|\wtf^m(J_0)| \ge \delta_5, $$
and so that~$\wtf^m$ is injective on~$J_0$.
It follows that~$f^m$ is injective on~$h^{-1}(J_0)$, so by the definition of~$\delta_5^*$ we have
$$ |J|
\ge
|h^{-1}(J_0)|
\ge
M^{-m} |h^{-1}(\wtf^m(J_0))|
\ge
\min \{|h(J)|^{A \ln M}, 1\} \cdot \delta_5^*. $$
This proves that~$h$ is H{\"o}lder continuous of exponent~$(A \ln M)^{-1}$.
\end{proof}

\section{Nonuniform hyperbolicity conditions}
\label{ss:proof of corollaries}
The purpose of this section is to prove
Corollaries~\ref{c:equivalences}, \ref{c:topological invariance of exponential mixing} and~\ref{c:phase transitions}.

\begin{proof}[Proof of Corollary~\ref{c:equivalences}]
To prove that conditions~$1$--$7$ are equivalent, remark first that the equivalence between conditions~$1$, $2$, $5$ and~$6$ is given by the Main Theorem', using Fact~\ref{f:positive entropy} for the implication~$5 \Rightarrow 6$.
When~$f$ is a complex rational map, the implication~$5 \Rightarrow 3$ is~\cite[Theorem~C]{PrzRiv07}.
The proof applies without change to the case where~$f$ is a non-degenerate smooth interval map that is topologically exact.\footnote{For a proof written for maps in~$\sA$, see~\cite[Corollary~$2.19$]{RivShe14}. If in addition~$f$ satisfies Collet-Eckmann condition and~$J(f) = I$, see also~\cite{KelNow92,You92} if~$f$ is unicritical, \cite{BruLuzvSt03} if all the critical points of~$f$ are of the same order and~\cite[Theorem~$6$]{GraSmi09} if~$f$ is real analytic.}
When~$f$ is unicritical, the implication~$3 \Rightarrow 2$ is~\cite[Lemma~$8.2$]{NowSan98}.
The proof applies without change to the general case.
We complete the proof that conditions~$1$--$6$ are equivalent by showing the implications~$5 \Rightarrow 4$ and~$4 \Rightarrow 2$.
For the implication~$5 \Rightarrow 4$, recall that by the general theory of Parry~\cite{Par66} and of Milnor and Thurston~\cite{MilThu88}, the map~$f$ is conjugated to a piecewise affine expanding map.
That the conjugacy is bi-H{\"o}lder follows from the combination of Lemma~\ref{l:affine model satisfies ESC} and Proposition~\ref{p:conjugacy to affine}.
When~$f$ is unicritical, the implication~$4 \Rightarrow 2$ is~\cite[Lemma~$8.4$]{NowSan98}.
The proof applies without change to the general case.
This completes the proof that conditions~$1$--$6$ are equivalent.

To complete the proof that conditions~$1$--$7$ are equivalent, we prove that condition~$7$ is equivalent to condition~$4$.
First notice that the conjugacy~$h \colon I \to [0, 1]$ to the piecewise affine model is H{\"o}lder continuous by Lemma~\ref{l:affine model satisfies ESC} and Proposition~\ref{p:conjugacy to affine}.
Thus condition~$4$ is equivalent to the condition that~$h^{-1}$ is H{\"o}lder continuous.
The conjugacy~$h$ is defined in terms of its unique maximal entropy measure~$\rho_f$, as follows: If we denote by~$a$ the left end point of~$I$, then for every~$x$ in~$I$ we have~$h(x) = \rho_f([a, x])$.
Thus, it readily follows that condition~$4$ is equivalent
condition~$7$.

To prove the final statement, note that the Backward Collet-Eckmann condition implies condition~$6$ trivially.
On the other hand, the Collet-Eckmann condition implies condition~$2$ by~\cite[Corollary~$1.1$]{BruvSt03}.
\end{proof}

\begin{rema}
\label{r:equivalences}
Conditions~$1$, $2$, $5$ and~$6$ of Corollary~\ref{c:equivalences} have natural formulations for maps in~$\sA$.
The Main Theorem' implies that, for maps that are essentially topologically exact on their Julia sets, these conditions are equivalent, using Fact~\ref{f:positive entropy} for the implication~$5 \Rightarrow 6$.
Using conformal measures, a condition analogous to condition~$3$ of Corollary~\ref{c:equivalences} can also be stated for a general interval map in~$\sA$.
Our results imply that in this more general setting condition~$3$ is equivalent to conditions~$1$, $2$, $5$ and~$6$.
In fact, the implication~$5 \Rightarrow 3$ is again given by either~\cite[Theorem~C]{PrzRiv07} or~\cite[Corollary~$2.19$]{RivShe14}.
The proof of the implication~$3 \Rightarrow 2$ for unicritical maps in~\cite[Lemma~$8.2$]{NowSan98} does not apply directly to this more general setting, as it uses that the reference measure is the Lebesgue measure.
Using Frostman's lemma, the argument can be adapted to deal with the case where the reference measure is a conformal measure, as in~\cite[Theorem~D]{PrzRiv07} for complex rational maps.
\end{rema}
\begin{rema}
\label{r:CE and CE2}
Both, the Collet-Eckmann and the Backward Collet-Eckmann condition have natural formulations for maps in~$\sA$.
In this more general setting each of these conditions implies conditions~$1$--$3$, $5$, and $6$ of Corollary~\ref{c:equivalences}, see Remark~\ref{r:equivalences}.
In fact, the Backward Collet-Eckmann condition implies condition~$6$ trivially and the Collet-Eckmann condition implies condition~$2$ by~\cite[Corollary~$1.1$]{BruvSt03}.
We note also that for a map in~$\sA$ the Collet-Eckmann condition implies the Backward Collet-Eckmann condition at each critical point of maximal order: For complex rational maps this is given by~\cite[Theorem~$1$]{GraSmi98}; the proof applies without change to maps in~$\sA$.\footnote{In fact, the proof for maps~$\sA$ is slightly simpler, as the arguments involving shrinking neighborhoods can be replaced by the one-sided Koebe principle.}
\end{rema}

\begin{proof}[Proof of Corollary~\ref{c:topological invariance of exponential mixing}]
We show that for a non-degenerate smooth map~${f \colon I \to I}$ having only
hyperbolic repelling periodic points, an iterate of~$f$ has an exponentially mixing acip if and only if:
\begin{enumerate}
\item[(*)]
There is an interval~$J$ contained in~$I$ and an integer~$s \ge 1$, such that~$f^s(J) \subset J$ and such that~$f^s \colon J \to J$ is a topologically exact map that satisfies the TCE condition.
\end{enumerate}
Since~(*) is clearly invariant under topological conjugacy preserving critical points, this implies the corollary.

If~(*) is satisfied, then~$f^s|_J$ is non-injective and therefore it is a non-degenerate smooth interval map.
Then Corollary~\ref{c:topological invariance} implies that~$f^s|_J$, and hence~$f^s$, has an exponentially mixing acip.

Suppose there is an integer~$s \ge 1$ such that~$f^s$ has an exponentially mixing acip~$\nu$, and denote by~$J$ the support of~$\nu$.
Then~$J$ is an interval, $f^s(J) \subset J$, and~$f^s|_J$ is topologically exact, see~\cite[Theorem~E($2$)]{vStVar04}.
It follows that~$f^s_J$ is non-injective and therefore that~$f^s|_J$ is a non-degenerate smooth interval map.
Thus Corollary~\ref{c:topological invariance} implies that~$f^s|_J$ satisfies the TCE condition.
This proves that~$f$ satisfies~(*), and completes the proof of the corollary.
\end{proof}

\begin{rema}
  \label{r:topological invariance of exponential mixing}
The proof of Corollary~\ref{c:topological invariance of exponential mixing} applies without change to maps in~$\sA$.
\end{rema}

\begin{proof}[Proof of Corollary~\ref{c:phase transitions}]
Denote by~$I$ the domain of~$f$.
Recall from~\S\ref{ss:phase transitions} that~$P$ in nonincreasing,
that it has at least one zero, and that its first zero~$t_0$ is in~$(0, 1]$.

The implication~$2 \Rightarrow 1$ is trivial, and the implication~$2
\Rightarrow 3$ is a direct consequence of the fact that~$P$ is
nonincreasing.
Since~$P$ has at least one zero, the implication~$3 \Rightarrow 2$ also follows from the fact that~$P$ is nonincreasing.

To prove the implication~$2 \Rightarrow 4$, suppose~$2$ holds.
Since the first zero of~$P$ is in~$(0, 1]$, we have~$P(2) = 0$.
So for each~$\chi > 0$ there is an ergodic measure~$\nu$ in~$\sM(I, f)$ satisfying~$h_\nu(f) - 2 \chi_\nu(f) \ge - \chi$.
By~\cite[Theorem~B]{Prz93} or Proposition~\ref{p:Lyapunov are nonnegative}, we have~$\chi_{\nu}(f) \ge 0$.
Combined with Ruelle's inequality
$$ h_{\nu}(f) \le \max \{0, \chi_{\nu}(f) \} = \chi_{\nu}(f), $$
see~\cite{Rue78}, we obtain
$$ 2 \chi_\nu(f) \le h_{\nu}(f) + \chi \le \chi_{\nu}(f) + \chi
\text{ and }
\chi_\nu(f) \le \chi. $$
Since~$\chi$ is arbitrary, this shows that~$\chiinff = 0$ and completes the proof of the implication~$2 \Rightarrow 4$.

To prove the implication~$4 \Rightarrow 3$, suppose~$\chiinf(f) = 0$,
and let~$t > t_0$ and~$\chi > 0$ be given.
Then there is a measure~$\nu$ in~$\sM(I, f)$ such that~$\chi_\nu(f) < \chi$, so
$$ P(t)
\ge
h_{\nu}(f) - t \chi_{\nu}(f)
\ge
- t\chi. $$
Since~$\chi > 0$ is arbitrary we conclude that~$P(t) \ge 0$ and hence that~$P$ is nonnegative.

We complete the proof of the corollary by showing the implication $1
\Rightarrow 4$.
Suppose~$\chiinf(f) > 0$, so that
$$ t_+
\=
\sup \{ t > 0 : P(t) > - t \chiinf(f) \} $$
satisfies~$t_+ > t_0$.
By~\cite[Theorem~A]{PrzRiv1405} the function~$P$ is real analytic
on~$(0, t_+)$, and hence at~$t = t_0$.
This proves that~$f$ does not have a high-temperature phase
transition, and completes the proof of the implication~$1 \Rightarrow 4$ and of the corollary.
\end{proof}

\begin{rema}
\label{r:phase transitions}
Each of the conditions~$1$--$4$ of Corollary~\ref{c:phase transitions} have natural formulations in the case where~$f$ is an interval map in~$\sA$.
The proof of Corollary~\ref{c:phase transitions} applies without change in this more general setting.
\end{rema}

\appendix
\section{Lyapunov exponents are nonnegative}
\label{s:Lyapunov are nonnegative}
In this appendix we prove the following general result characterizing those invariant measures whose Lyapunov exponent is strictly negative (possibly infinite).
For smooth interval maps with a finite number of non-flat critical points, this was shown by Przytycki in~\cite[Theorem~B]{Prz93}.
We give a proof of this important fact that avoids the Koebe principle and applies to continuously differentiable maps.
It is considerably shorter than the proof in~\cite{Prz93} and extends without change to complex rational maps.

For a continuously differentiable interval map~$f$, a periodic
orbit of~$f$ of period~$n$ is \emph{strictly attracting}, if for
each point~$p$ in this orbit~$|Df^n(p)| < 1$.
For a Borel measure~$\nu$ on a topological space~$X$, we use~$\supp(\nu)$ to denote the support of~$\nu$, which is by definition the set of all points in~$X$ such that the measure of each of its neighborhoods is strictly positive.
\begin{prop}
  \label{p:Lyapunov are nonnegative}
Let~$f$ be a continuously differentiable interval map and let~$\nu$ be an ergodic invariant probability measure.
Then either~$\chi_{\nu}(f) \ge 0$ or~$\nu$ is supported on a strictly attracting periodic orbit of~$f$.
\end{prop}
\begin{proof}
Suppose~$\chi_{\nu}(f) < 0$.
By the dominated convergence theorem there exists~$L > 0$ such that the function
$$
\varphi \= \max \{ \ln |Df|, - L \}
$$
satisfies $A \= \int \varphi \dd \nu < 0$.
Fix~$\chi$ in~$(0, - A/3)$ and for each integer~$n \ge 1$ put
$$ S_n(\varphi)
\=
\varphi + \varphi \circ f + \cdots + \varphi \circ f^{n - 1}. $$

\partn{1}
We show that for every point~$x$ in the domain~$I$ of~$f$ satisfying
$$ \lim_{n \to + \infty} \tfrac{1}{n} S_n(\varphi)(x) = A, $$
there exists $\tau > 0$ such that for every sufficiently large integer~$n$ we have $|Df^n| \le \exp(-\chi n)$ on $B(x, \tau)$.
Fix such~$x$ in~$I$ and let $\delta > 0$ be such that we have $|Df| \le \exp(-L)$ on $B(\Crit(f), \delta)$.
As~$f$ is continuously differentiable there is~$\varepsilon$ in~$(0, \delta/3)$ such that the distortion of~$f$ on an interval of length at most~$\varepsilon$ and disjoint from $B(\Crit(f), \delta/3)$ is at most~$\exp(\chi)$.
By our choice of~$\chi$ there is $\tau > 0$ so that for every $n \ge 0$ we have
$$
\tau \exp( S_n(\varphi)(x) + 3 n \chi) < \varepsilon / 2.
$$
Finally, for each $n \ge 0$ put
$$
r_n \= \tau \exp( S_n(\varphi)(x) + n\chi)
\text{ and }
B_n \= B(f^n(x), r_n).
$$
Note that we have $|B_n| = 2r_n \le \varepsilon \exp(- 2 n\chi)$.

We show that for every $n \ge 0$ we have $|Df| \le \exp( \varphi(f^n(x)) + \chi)$ on~$B_n$. 
This implies that $f(B_n) \subset B_{n + 1}$ and by induction that on~$B(x, \tau)$ we have
$$
|Df^n|
\le
\exp(S_n(\varphi)(x) + \chi n)
\le
\tau^{-1} (\varepsilon / 2) \exp(-2 n \chi).
$$
It then follows that for large~$n$ we have $|Df^n| \le \exp(-\chi n)$ on~$B(x, \tau)$, as wanted.

\partn{Case 1}
$f^n(x) \not \in B(\Crit(f), 2\delta/3)$.
Since the length of~$B_n$ is less than~$\varepsilon < \delta /3$, it follows that the interval~$B_n$ is disjoint from~$B(\Crit(f), \delta/3)$ and that the distortion of~$f$ on~$B_n$ is bounded by $\exp(\chi)$.
So on~$B_n$ we have
$$
|Df|
\le
|Df(f^n(x))| \exp(\chi)
\le
\exp(\varphi(f^n(x)) + \chi).
$$

\partn{Case 2}
$f^n(x) \in B(\Crit(f), 2\delta/3)$.
Then $B_n \subset B(\Crit(f), \delta)$ and by our choice of~$\delta$  we have $|Df| \le \exp(-L)$ on~$B_n$.

\partn{2}
By Birkhoff's ergodic theorem the set of points~$x$ satisfying the property described in part~$1$ has full measure with respect to~$\nu$.
We can thus find such a point~$x$ in~$\supp(\nu)$, such that in addition its orbit is dense in~$\supp(\nu)$.
Let~$\tau > 0$ be given by the property described in part~$1$ for this choice of~$x$.
Then there is an integer~$n \ge 1$ such that $|Df^n| \le \exp(-n\chi) \le \tfrac{1}{4}$ on~$B(x, \tau)$ and such that~$f^n(x)$ is in~$B(x, \tau/4)$.
Then
$$ f^n(B(x, \tau)) \subset B(f^n(x), \tau/2) $$
and~$f^n$ is uniformly contracting on~$B(x, \tau)$.
This implies that~$x$ is asymptotic to a strictly attracting periodic point of~$f$.
Since~$x$ is in~$\supp(\nu)$ and~$\nu$ is ergodic, it follows
that~$\nu$ is supported on a strictly attracting periodic orbit of~$f$.
\end{proof}

\bibliographystyle{alpha}
% \bibliography{$HOME/papers/_bib/papers}

\end{document}